\documentclass[a4paper]{article}
\usepackage{amsmath, amsfonts, amsthm, amssymb, amscd}
\usepackage{latexsym}
\usepackage{graphicx, mathptmx}
\usepackage{mathtools, stackengine, mleftright}
\usepackage[mathcal]{euscript} 
\usepackage{newtxtext, newtxmath, dsfont, mathrsfs, accents, verbatim} 
\usepackage{enumerate, color}
\usepackage{epsfig}
\usepackage{delarray}
\usepackage[titletoc,title]{appendix}
\usepackage{color}
\newtheorem{theorem}{Theorem}
\newtheorem{definition}{Definition}
\newtheorem{remark}{Remark}
\newtheorem{lemma}{Lemma}
\newtheorem{proposition}{Proposition}

\numberwithin{equation}{section}

\newcommand{\divx}{\mathop{\mathrm{div}}}

\newcommand{\supp}{\mathop{\mathrm{supp}}}
\newcommand{\esssup}{\mathop{\mathrm{ess~sup}}}
\newcommand{\Lip}{\mathop{\mathrm{Lip}}}

\newcommand{\dist}{\mathop{\mathrm{dist}}}
\newcommand{\trace}{\mathop{\mathrm{tr}}}
\allowdisplaybreaks[2]
\title{Continuity of derivatives of a convex solution to a perturbed one-Laplace equation by $p$-Laplacian}
\author{Yoshikazu Giga\footnote{Graduate School of Mathematical Sciences, The University of Tokyo, Japan. \textit{Email}: \texttt{labgiga@ms.u-tokyo.ac.jp}}\quad and  \quad
        Shuntaro Tsubouchi  \footnote{Graduate School of Mathematical Sciences, The University of Tokyo, Japan. \textit{Email}: \texttt{tsubos@ms.u-tokyo.ac.jp}} }

\begin{document}
\maketitle
\begin{abstract}
We consider a one-Laplace equation perturbed by $p$-Laplacian with $1<p<\infty$.
 We prove that a weak solution is continuously differentiable $(C^1)$ if it is convex.
 Note that similar result fails to hold for the unperturbed one-Laplace equation.
 The main difficulty is to show $C^1$-regularity of the solution at the boundary of a facet where the gradient of the solution vanishes.
 For this purpose we blow-up the solution and prove that its limit is a constant function by establishing a Liouville-type result, which is proved by showing a strong maximum principle. Our argument is rather elementary since we assume that the solution is convex. A few generalization is also discussed.
\end{abstract}
\bigbreak

\section{Introduction}
We consider a one-Laplace equation perturbed by $p$-Laplacian of the form
\begin{equation}\label{crystal model eq again}
L_{b,p} u = f \quad\text{in}\quad \Omega
\end{equation}
with
\[
	L_{b,p} u := -b\Delta_{1} u - \Delta_{p} u,
\]
where
\[
	\Delta_1 u := \operatorname{div} \left( \nabla u/|\nabla u| \right), \quad
	\Delta_p u = \operatorname{div} \left( |\nabla u|^{p-2} \nabla u \right)
\]
in a domain $\Omega$ in $\mathbb{R}^n$, \(\nabla u=(\partial_{x_{1}}u,\,\dots,\,\partial_{x_{n}}u)\) with \(\partial_{x_{j}}u=\partial u/\partial x_{j}\) for a function \(u=u(x_{1},\,\dots,\,x_{n})\), and \(\divx X=\sum\limits_{i=1}^{n}\partial_{x_{i}}X_{i}\) for a vector field \(X=(X_{1},\,\dots,\,X_{n})\).
 The constants $b>0$ and $p\in(1,\infty)$ are given and fixed.
 It has been a long-standing open problem whether its weak solution is $C^1$ up to a facet, the place where the gradient $\nabla u$ vanishes, even if $f$ is smooth.
 This is a non-trivial question since a weak solution to the unperturbed one-Laplace equation, i.e., $-\Delta_1 u=f$ may not be $C^1$.
 This is because the ellipticity degenerates in the direction of $\nabla u$ for $\Delta_1 u$.
 Our goal in this paper is to solve this open problem under the assumption that a solution is convex.
 \subsection{Main theorems and our strategy}\label{Subsect Main theorems}
Throughout the paper, we assume $f\in L^q_{\mathrm{loc}}(\Omega)$ ($n<q\leq\infty$), i.e., $|f|^q$ is locally integrable in $\Omega$.
 Our main result is
\begin{theorem}[\(C^{1}\)-regularity theorem]\label{C1 regularity}
Let $u$ be a convex weak solution to (\ref{crystal model eq again}) with $f\in L^q_{\mathrm{loc}}(\Omega)\,(n<q\le\infty)$. Then $u$ is in $C^{1}(\Omega)$.
\end{theorem}

Difficulty on proving regularity on gradients of solutions to (\ref{crystal model eq again}) can be explained from a viewpoint of ellipticity ratio.  We set a convex function $E\colon {\mathbb R}^{n}\rightarrow [0,\,\infty)$ by
\[E(z)\coloneqq bE_{1}(z)+E_{p}(z)\quad\textrm{for } z\in{\mathbb R}^{n},\]
where \(E_{s}\,(1\le s<\infty)\) is defined by
\[E_{s}\coloneqq \frac{1}{s}\lvert z\rvert^{s}\quad\textrm{for } z\in{\mathbb R}^{n}.\]
We rewrite (\ref{crystal model eq again}) by 
\begin{equation}\label{equation in divergence form}
-\divx(\nabla_{z}E(\nabla u))=f\quad\textrm{in }\Omega.
\end{equation}
By differentiating (\ref{equation in divergence form}) by \(x_{i}\,(i\in\{\,1,\,\dots,\,n\,\})\), we get
\begin{equation}\label{equation differentiated; heuristic}
-\divx\mleft(\nabla_{z}^{2}E(\nabla u)\nabla \partial_{x_{i}}u\mright)=\partial_{x_{i}}f.
\end{equation}
By elementary calculations, ellipticity ratio of the Hessian $\nabla_{z}^{2}E$ at $z_{0}\in {\mathbb R}^{n}\setminus\{ 0\}$ is given by
\begin{align*}
\mleft(\textrm{ellipticity ratio of \(\nabla_{z}^{2}E(z_{0})\)}\mright) &\coloneqq \frac{(\textrm{the largest eigenvalue of \(\nabla_{z}^{2}E(z_{0})\)})}{(\textrm{the lowest eigenvalue of \(\nabla_{z}^{2}E(z_{0})\))}}\\
&=\frac{\max(p-1,\,1)+b\lvert z_{0}\rvert^{1-p}}{\min(p-1,\,1)}.
\end{align*}
Since the exponent $1-p$ is negative, the ellipticity ratio of $\nabla_{z}^{2}E(z_{0})$ blows up as $z_{0}\to 0$. By this property, we can observe that the equation (\ref{equation in divergence form}) becomes non-uniformly elliptic near the facet. It should be noted that our problem is substantially different from the $(p,\,q)$-growth problem, since for $(p,\,q)$-growth equations, non-uniform ellipticity appears as a norm of a gradient blows up \cite[Section 6.2]{MR2291779}. Although regularity of minimizers of double phase functionals, including
\[{\mathcal H}(u)\coloneqq \int E_{p}(\nabla u) \,dx+\int a(x) E_{q}(\nabla u)\,dx  \quad \textrm{with}\quad 1<p\le q<\infty,\,a(x)\ge 0\]
were discussed in scalar and even in vectorial cases by Colombo and Mingione \cite{MR3360738,MR3294408}, their results do not recover our $C^{1}$-regularity results.
This is basically derived from the fact that, unlike \(\nabla_{z}^{2}E_{p}\) with \(1<p<\infty\), the Hessian matrix \(\nabla_{z}^{2}E_{1}(z_{0})\,(z_{0}\not=0)\) always takes \(0\) as its eigenvalue. In other words, ellipticity of the operator \(\Delta_{1}u\) degenerates in the direction of \(\nabla u\), which seems to be difficult to handle analytically.

On the other hand, the ellipticity ratio of $\nabla_{z}^{2}E(z_{0})$ is uniformly bounded over $\lvert z_{0}\rvert>\delta$ for each fixed $\delta>0$.
In this sense we may regard the equation (\ref{equation differentiated; heuristic}) as locally uniformly elliptic outside the facet.
To show Lipschitz bound, we do not need to study over the facet. In fact, local Lipschitz continuity of solutions to (\ref{crystal model eq again}) are already established in \cite{MR4201656}; see also \cite{MR4148404} for a weaker result.
To study continuity of derivatives, we have to study regularity up to the facet. Thus, it seems to be impossible to apply standard arguments based on De Giorgi--Nash--Moser theory. In this paper, we would like to show continuity of derivatives of convex solutions by elementary arguments based on convex analysis.

Let us give a basic strategy to prove Theorem \ref{C1 regularity}. Since the problem is local, we may assume that \(\Omega\) is convex, or even a ball. By $C^{1}$-regularity criterion for a convex function, to show \(u\) is \(C^{1}\) at \(x\in\Omega\) it suffices to prove that
\begin{equation}\label{reduced question}
\textrm{the subdifferential \(\partial u(x)\) at \(x\in\Omega\) is a singleton;}
\end{equation}
see \cite[Appendix D]{MR3887613}, \cite[\S 25]{MR1451876} and Remark \ref{Some properties on convex function} for more detail. Here the subdifferential of \(u\) at \(x_{0}\in\Omega\) is defined by 
\[\partial u(x_{0})\coloneqq \mleft\{z\in{\mathbb R}^{n}\mathrel{}\middle|\mathrel{}u(x)\ge u(x_{0})+\langle z\mid x-x_{0}\rangle\textrm{ for all }x\in\Omega\mright\}.\]
Here \(\langle\,\cdot\,\mid\,\cdot\,\rangle\) stands for the standard inner product in \({\mathbb R}^{n}\).
For a convex function \(u\colon\Omega\rightarrow{\mathbb R}\), we can simply express the facet of \(u\) as
\[F\coloneqq\{x\in\Omega\mid\partial u(x)\ni 0\}=\{x\in\Omega\mid u(x)\le u(y)\textrm{ for all }y\in\Omega\}.\]
By definition it is clear that the facet \(F\) is non-empty if and only if a minimum of $u$ in \(\Omega\) exists. By convexity of \(u\), we can easily check that \(F\subset \Omega\) is a relatively closed convex set in \(\Omega\). We also define an open set
\[D\coloneqq \Omega\setminus F=\{x\in\Omega\mid u(y)<u(x)\textrm{ for some }y\in\Omega\}.\]
Our strategy to show (\ref{reduced question}) depends on whether \(x\) is inside \(F\) or not.

\begin{remark}\upshape[Some properties on differentiability of convex functions]\label{Some properties on convex function}
Let $v$ a real-valued convex function in a convex domain \(\Omega\subset{\mathbb R}^{n}\), then following properties hold.
\begin{enumerate}
\item $v$ is locally Lipschitz continuous in $\Omega$, and therefore $v$ is a.e. differentiable in $\Omega$ by Rademacher's theorem (\cite[Theorem 1.19]{MR3887613}, see also \cite[Theorem 3.1 and 3.2]{Evans-Gariepy} and \cite[Theorem 25.5]{MR1451876}).
\item For $x\in \Omega$, \(v\) is differentiable at \(x\) if and only if the subdifferential set \(\partial v(x)\) is a singleton. Moreover, if \(x\in\Omega\) satisfies either of these equivalent conditions, then we have \(\partial v(x)=\{\nabla v(x)\}\) (\cite[Proposition D.5]{MR3887613}, see also \cite[Theorem 25.1]{MR1451876}). In particular, Rademacher's theorem implies that $\partial v(x)=\{\nabla v(x)\}$ for a.e. $x\in\Omega$.
\item $v\in C^{1}(\Omega)$ if and only if $\partial v$ is single-valued (\cite[Remark D.3 (iii)]{MR3887613}, see also \cite[Theorem 25.5]{MR1451876}).
\end{enumerate}
Throughout this paper, we use these well-known results without proofs.
\end{remark}

We first discuss the case \(x\in D\).
Our goal is to show directly that \(u\) is \(C^{1,\,\alpha}\) near a neighborhood of \(x\) and therefore \(\partial u(x)=\{\nabla u(x)\}\not= \{0\}\) for all \(x\in D\). This strategy roughly consists of three steps.
Among them the first step, a kind of separation of $x\in D$ from the facet $F$, plays an important role.
Precisely speaking, we first find a neighborhood $B_{r}(x)\subset D$, an open ball centered at $x$ with its radius $r>0$, such that 
\begin{equation}\label{key method in D}
\partial_{\nu}u\ge \mu>0\quad\textrm{a.e. in}\quad B_{r}(x)
\end{equation}
for some direction $\nu$ and some constant $\mu>0$.
In order to justify (\ref{key method in D}), we fully make use of convexity of $u$ (Lemma \ref{coercivity lemma for convex function} in Section \ref{Appendix A}), not elliptic regularity theory.
Then with the aid of local Lipschitz continuity of $u$, the inclusion \(B_{r}(x)\subset \{0<\mu\le \partial_{\nu}u\le \lvert\nabla u\rvert\le M\}\) holds for some finite positive constant $M$.
Secondly, this inclusion allows us to check that \(u\) admits local \(W^{2,\,2}\)-regularity in $B_{r}(x)$ by the standard difference quotient method. Therefore we are able to obtain the equation (\ref{equation differentiated; heuristic}) in the distributional sense.
Finally, we appeal to the classical De Giorgi--Nash--Moser theory to obtain local \(C^{1,\,\alpha}\)-regularity at \(x\in D\), since the equation (\ref{equation differentiated; heuristic}) is uniformly elliptic in $B_{r}(x)$. Here the constant \(\alpha\in(0,\,1)\) we have obtained may depend on the location of \(x\in D\) through ellipticity, so \(\alpha\) may tend to zero as \(x\) tends to the facet.

It takes much efforts to prove that \(\partial u(x)=\{0\}\) for all \(x\in F\). Our strategy for justifying this roughly consists of three parts; a blow-argument for solutions, a strong maximum principle, and a Liouville-type theorem. Here we describe each individual step.

We first make a blow-argument. Precisely speaking, for a given convex solution \(u\colon\Omega\rightarrow{\mathbb R}\) and a point \(x_{0}\in \Omega\), we set a sequence of rescaled functions \(\{u_{a}\}_{a>0}\) defined by \[u_{a}(x)\coloneqq \frac{u(a(x-x_{0})+x_{0})-u(x_{0})}{a}.\]
We show that \(u_{a}\) locally uniformly converges to some convex function \(u_{0}\colon{\mathbb R}^{n}\rightarrow{\mathbb R}\), which satisfies $\partial u(x_{0})\subset \partial u_{0}(x_{0})$ by construction. Moreover, we prove that $u_{0}$ satisfies $L_{b,\,p}u_{0}=0$ in ${\mathbb R}^{n}$ in the distributional sense.
There we will face to justify a.e. convergence of gradients, and this is elementarily shown by regarding gradients in the classical sense as subgradients (Lemma \ref{a.e. pointwise convergence for convex functions} in the appendices).

Next we prove that if $x_{0}\in F$, then the convex weak solution \(u_{0}\) constructed as above satisfies \(\partial u_{0}(x_{0})=\{0\}\).
Moreover, we are going to prove that \(u_{0}\) is constant (a Liouville-type theorem).
For this purpose we establish the maximum principle.
\begin{theorem}[Strong maximum principle]\label{strong maximum principle}
Let \(u\) be a convex weak solution to \(L_{b,\,p}u=0\) in a convex domain \(\Omega\subset{\mathbb R}^{n}\) and $F\subset\Omega$ be the facet of $u$. Then \(u\) is affine in each connected component of the open set \(D\coloneqq \Omega\setminus F\). In particular, if $F=\emptyset$, then $u$ is affine in $\Omega$.
\end{theorem}
It should be noted that this result is a kind of strong maximum principle in the sense that
\begin{equation}\label{strong comparison principle}
u\ge a\textrm{ in }D_{0}\textrm{ and }u(x_{0})=a(x_{0})\textrm{ for }x_{0}\in D_{0}\quad\textrm{imply that}\quad u\equiv a\textrm{ in }D_{0},
\end{equation}
where \(a(x)\coloneqq u(x_{0})+\langle\nabla u(x_{0})\mid x-x_{0}\rangle\) and \(D_{0}\) is a connected component of \(D\). The affine function \(a\) clearly satisfies \(L_{b,\,p}a=0\) in the classical sense. 

In order to justify (\ref{strong comparison principle}), we will face three problems.
The first is a justification of the comparison principle, the second is regularity of \(u\), and the third is a construction of suitable barrier subsolutions, all of which are essentially needed in the classical proof of E. Hopf's strong maximum principle \cite{MR50126}.
In order to overcome these obstacles, we appeal to both classical and distributional approaches, and restrict our analysis only over regular points. For details, see Section \ref{Subsect Literature Overview on maximum principles}.

Even though our strong maximum principle is somewhat weakened in the sense that this holds only on each connected component of \(D\subset\Omega\), we are able to show the following Liouville-type theorem.
\begin{theorem}[Liouville-type theorem]\label{Liouville theorem}
Let \(u\) be a convex weak solution to \(L_{b,\,p}u=0\) in \({\mathbb R}^{n}\). Then $F\subset {\mathbb R}^{n}$, the facet of $u$, satisfies either $F=\emptyset$ or $F={\mathbb R}^{n}$. In particular, $u$ satisfies either of the followings.
\begin{enumerate}
\item If $u$ attains its minimum in ${\mathbb R}^{n}$, then $u$ is constant.
\item If $u$ does not attains its minimum in ${\mathbb R}^{n}$, then $u$ is a non-constant affine function in ${\mathbb R}^{n}$.
\end{enumerate}
\end{theorem}
In the proof of the Liouville-type theorem, our strong maximum principle plays an important role. Precisely speaking, if a convex solution in the total space does not satisfy $\emptyset\subsetneq F\subsetneq {\mathbb R}^{n}$, then Theorem \ref{strong maximum principle} and the supporting hyperplane theorem from convex analysis help us to determine the shape of convex solutions. In particular, the convex solution can be classified into three types of piecewise-linear functions of one-variable.
These non-smooth piecewise-linear functions are, however, no longer weak solutions, which we will prove by some explicit calculations.

By applying the Liouville-type theorem and our blow-argument, we are able to show that subgradients at points of the facet are always \(0\), i.e., \(\partial u(x)=\{0\}\) for all \(x\in F\), and we complete the proof of the \(C^{1}\)-regularity theorem. Note that the statements in Theorem \ref{strong maximum principle} and \ref{Liouville theorem} should not hold for unperturbed one-Laplace equation $-\Delta_{1}u=f$, since any absolutely continuous non-decreasing function of one variable $u=u(x_{1})$ satisfies $-\Delta_{1}u=0$.

Finally we mention that we are able to refine our strategy, and obtain $C^{1}$-regularity of convex solutions to more general equations. We replace the one-Laplacian $\Delta_{1}$ by another operator which is derived from a general convex functional of degree $1$. This generalization requires us to modify some of our arguments, including a blow-up argument and the Liouville-type theorem. For further details, see Section \ref{Subsect Outline} and Section \ref{Subsect Sketch}.

\subsection{Literature overview on maximum principles}\label{Subsect Literature Overview on maximum principles}
We briefly introduce maximum principles related to the paper. We also describe our strategy to establish the strong maximum principle.

Maximum principles, including comparison principles and strong maximum principles, have been discussed by many mathematicians in various settings. In the classical settings, E. Hopf proved a variety of maximum principles on elliptic partial differential equations of second order, by elementary arguments based on constructions of auxiliary functions. E. Hopf's strong maximum principle is one of the well-known results on maximum principles. In Hopf's proof of the strong maximum principle \cite{MR50126}, he defined an auxiliary function
\begin{equation}\label{Hopf comparison ft}
h(x)\coloneqq e^{-\alpha\lvert x-x_{\ast}\rvert^{2}}-e^{-\alpha R^{2}}\quad\textrm{for }x\in{\mathbb R}^{n},
\end{equation}
which becomes a classical subsolution in a fixed open annulus \(E_{R}=E_{R}(x_{\ast})\coloneqq B_{R}(x_{\ast})\setminus\overline{B_{R/2}(x_{\ast})}\) for sufficiently large \(\alpha>0\).
An alternative function
\begin{equation}\label{alternative choice of h}
h(x)\coloneqq \lvert x-x_{\ast}\rvert^{-\alpha}-R^{-\alpha}\quad\textrm{for }x\in{\mathbb R}^{n}\setminus\{x_{\ast}\}
\end{equation}
is given in \cite[Chapter 2.8]{MR2356201}.
E. Hopf's classical results on maximum principles are extensively contained in \cite[Chapter 3]{MR1814364}, \cite[Chapter 2]{MR0219861} and \cite[Chapter 2]{MR2356201}.

The materials \cite[Chapter 8--9]{MR1814364} and \cite[Chapter 3--6]{MR2356201} provide proofs of maximum principles, including strong maximum principles, even for distributional solutions. Among them, \cite[Theorem 5.4.1]{MR2356201} deals with a justification of the strong maximum principle for distributional supersolutons to certain quasilinear elliptic equations with divergence structures,
\[\textrm{i.e., }-\divx(A(x,\,\nabla u(x)))=0,\]
which covers the \(p\)-Laplace equation with \(1<p<\infty\).
Even in the distributional schemes, the proof of the maximum principle \cite[Theorem 5.4.1]{MR2356201} is partially similar to E. Hopf's classical one, in the sense that it is completed by calculating directional derivatives of auxiliary functions. The significant difference is, however, the construction of spherically symmetric subsolutions of \(C^{1}\) class, which is given in \cite[Chapter 4]{MR2356201}, is based on Leray--Schauder's fixed point theorem \cite[Theorem 11.6]{MR1814364}.
Also it should be noted that the proofs of comparison principles \cite[Theorem 2.4.1 and 3.4.1]{MR2356201} are just based on strict monotonicity of the mapping \(A(x,\,\cdot\,)\colon{\mathbb R}^{n}\rightarrow {\mathbb R}^{n}\), whereas Hopf's proof appeals to direct constructions of auxiliary functions.

With our literature overview in mind, we describe our strategy for showing (\ref{strong comparison principle}).
A justification of comparison principles is easily obtained in the distributional schemes (see \cite[Chapter 3]{MR2356201} as a related material).
However, the remaining two obstacles, the differentiability of \(u\) and the construction of subsolutions, cannot be resolved affirmatively by just imitating arguments given in \cite[Chapter 4--5]{MR2356201}.
In the first place, it should be mentioned that convex weak solutions we treat in this paper are assumed to have only local Lipschitz regularity, whereas supersolutions treated in \cite[Chapter 5]{MR2356201} are required to be in \(C^{1}\).
We recall that \(C^{1}\)-regularity of convex weak solutions can be guaranteed in \(D\subset\Omega\) (the outside of the facet) by the classical De Giorgi--Nash--Moser theory, and this result enables us to overcome the problem whether \(u\) is differentiable at certain points.
This is the reason why Theorem \ref{strong maximum principle} need to restrict on $D$.
Although the construction of distributional subsolutions is generally discussed in \cite[Chapter 4]{MR2356201}, we do not appeal to this.
Instead, we directly construct a function \(v=\beta h+a\) in \({\mathbb R}^{n}\setminus\{x_{\ast}\}\), where \(\beta>0\) is a constant and \(h\) is defined as in (\ref{Hopf comparison ft}) or (\ref{alternative choice of h}).
We will determine the constants \(\alpha,\,\beta>0\) so precisely that \(v\) satisfies \(L_{b,\,p}v\le 0\) in the classical sense over a fixed open annulus \(E_{R}=E_{R}(x_{\ast})\).
We also make \(\lvert \nabla v\rvert\) very close to \(\lvert \nabla a\rvert\equiv\lvert\nabla u(x_{0})\rvert>0\) over \(E_{R}\), so that \(\nabla v\) no longer degenerates there.
By direct calculation of \(L_{b,\,p}v\), we explicitly construct classical subsolutions to \(L_{b,\,p}u=0\) in \(E_{R}\).
Finally we are able to deduce (\ref{strong comparison principle}) by an indirect proof.

Another type of definitions of subsolutions and supersolutions to (\ref{crystal model eq again}) in the distributional schemes can be found in F. Kr\"{u}gel's thesis in 2013 \cite{krugel2013variational}. The significant difference is that Kr\"{u}gel did not regard the term \(\nabla u/\lvert \nabla u\rvert\) as a subgradient vector field. Since monotonicity of \(\partial\lvert\,\cdot\,\rvert\) is not used at all, it seems that Kr\"{u}gel's proof of comparison principle \cite[Theorem 4.8]{krugel2013variational} needs further explanation. For details, see Remark \ref{Kruegel mistakes}.
\subsection{Mathematical models and previous researches}
Our problem is derived from a minimizing problem of a certain energy functional, which involves the total variation energy.
The equation (\ref{crystal model eq again}) is deduced from the following Euler--Lagrange equation;
\[f=\frac{\delta G}{\delta u},\quad \textrm{where}\quad G(u)\coloneqq b\int_{\Omega}\lvert \nabla u\rvert\,dx+\frac{1}{p}\int_{\Omega}\lvert\nabla u\rvert^{p}\,dx.\]
The energy functional \(G\) often appears in fields of materials science and fluid mechanics.

In \cite{spohn1993surface}, Spohn modeled the relaxation dynamics of a crystal surface below the roughening temperature. On $h$ describing the height of the crystal for a two-dimensional domain $\Omega$ is modeled as
\[h_{t}+\divx j=0\]
with $j=-\nabla \mu$, where \(\mu\) is a chemical potential.
In \cite{spohn1993surface}, its evolution is given as
\[\mu=\frac{\delta \Phi}{\delta h}\quad\textrm{with}\quad \Phi(h)=\int_{\Omega}\lvert\nabla h\rvert\,dx+\kappa \int_{\Omega}\lvert\nabla h\rvert^{3}\,dx\]
with \(\kappa>0\). This \(\Phi\) is essentially the same as \(G\) with \(p=3\).
Then, the resulting evolution equation for \(h\) is of the form
\[bh_{t}=\Delta L_{b,\,3}h\quad\textrm{with}\quad b=\frac{1}{3\kappa}.\]
This equation can be defined as a limit of step motion, which is microscopic in the direction of height \cite{MR3289366}; see also \cite{odisharia2006simulation}.
The initial value problem of this equation can be solved based on the theory of maximal monotone operators \cite{MR2746654} under the periodic boundary condition. 
Subdifferentials describing the evolution are characterized by Kashima \cite{MR2083617}, \cite{MR2885966}. Its evolution speed is calculated by \cite{MR2083617} for one dimensional setting and by \cite{MR2885966} for radial setting. It is known that the solution stops in finite time \cite{MR2772127}, \cite{MR3362778}.
In \cite{odisharia2006simulation}, numerical calculation based on step motion is calculated. If one considers a stationary solution, $h$ must satisfies
\[\Delta L_{b,\,3}h=0.\]
If \(L_{b,\,3}h\) is a constant, our Theorem \ref{C1 regularity} implies that the height function $h$ is $C^{1}$ provided that $h$ is convex.

For a second order problem, \[\textrm{i.e., }bh_{t}=L_{b,\,p}h,\] its analytic formulation goes back to \cite{MR0394323}, \cite[Chapter VI]{MR0521262}  for $p=2$, and its numerical analysis is given in \cite{MR635927}.
For the fourth order problem, its numerical study is more recent. The reader is referred to papers by \cite{MR3902521}, \cite{MR4045243}, \cite{MR2733098}. 

Another important mathematical model for the equation (\ref{crystal model eq again}) is found in fluid mechanics. Especially for \(p=2\) and \(n=2\), the energy functional \(G\) appears when modeling stationary laminar incompressible flows of a material called Bingham fluid, which is a typical non Newtonian fluid. Bingham fluid reflects the effect of plasticity  corresponding to $\Delta_{1}u$ as well as that of viscosity corresponding $\Delta_{2} u =\Delta u$ in (\ref{crystal model eq again}). 
Let us consider a parallel stationary flow with velocity $U=(0,0, u(x_1, x_2))$ in a cylinder $\Omega\times {\mathbb R}$. Of course, this is incompressible flow, i.e., $\divx U=0$. If this flow is the classical Newtonian fluid, then the Navier--Stokes equations become (\ref{crystal model eq again}) in $\Omega$ with $b=0$ and $f=-\partial_{x_3} \pi$, where $\pi$ denotes the pressure. In the case that plasticity effects appears, one obtains (\ref{crystal model eq again}), following \cite[Chapter VI, Section 1]{MR0521262}. There it is also mentioned that since the velocity is assumed to be uni-directional, the external force term in (\ref{crystal model eq again}) is considered as constant in this laminar flow model.
The significant difference is that motion of the Bingham fluid is blocked if the stress of the Bingham fluid exceeds a certain threshold.
This physical phenomenon is essentially explained by the nonlinear term \(b\Delta_{1}u\), which reflects rigidity of the Bingham fluid.
For more details, see \cite[Chapter VI]{MR0521262} and the references therein.

On continuity of derivatives for solutions, less is known even for the second order elliptic case. Although Kr\"{u}gel gave an observation that solutions can be continuously differentiable \cite[Theorem 1.2]{krugel2013variational} on the boundary of a facet, mathematical justifications of \(C^{1}\)-regularity have not been well-understood. Our main result (Theorem \ref{C1 regularity}) mathematically establishes continuity of gradient for convex solutions.

\subsection{Organization of the paper}\label{Subsect Outline}
We outline the contents of the paper.

Section \ref{Sect Regularity outside facets} establishes \(C^{1,\,\alpha}\)-regularity at regular points of convex weak solutions (Lemma \ref{regularity at regular point}).
In order to apply De Giorgi--Nash--Moser theory, we will need to justify local \(W^{2,\,2}\)-regularity by the difference quotient method.
The key lemma, which is proved by convex analysis, is contained in the appendices (Lemma \ref{coercivity lemma for convex function}).

Section \ref{Sect Blow-argument} provides a blow-up argument for convex weak solutions.
The aim of Section \ref{Sect Blow-argument} is to prove that \(u_{0}\colon{\mathbb R}^{n}\rightarrow{\mathbb R}\), a limit of rescaled solutions, satisfies \(L_{b,\,p}u_{0}=0\) in the weak sense over the whole space ${\mathbb R}^{n}$ (Proposition \ref{limit of rescaled solutions}).
To assure this, we will make use of an elementary result on a.e. convergence of gradients, which is given in the appendices (Lemma \ref{a.e. pointwise convergence for convex functions}).

Section \ref{Sect Maximum Principles} is devoted to justifications of maximum principles for the equation $L_{b,\,p}u=0$. We first give definitions of sub- and supersolution in the weak sense. Section \ref{Subsect Comparison Principle} provides a justification of the comparison principle (Proposition \ref{comparison principle}). Section \ref{Subsect Construction of subsolutions} establishes an existence result of classical barrier subsolutions in an open annulus (Lemma \ref{Hopf Lemma}). Applying these results in Section \ref{Subsect Comparison Principle}--\ref{Subsect Construction of subsolutions}, we prove the strong maximum principle outside the facet (Theorem \ref{strong maximum principle}).

In Section \ref{Sect Liouville Theorem}, we will show the Liouville-type theorem (Theorem \ref{Liouville theorem}) by making use of Theorem \ref{strong maximum principle}, and complete the proof of our main theorem (Theorem \ref{C1 regularity}).

Finally in Section \ref{Sect Generalization}, we discuss a few generalization of the operators $\Delta_{1}$ and $\Delta_{p}$. Since the general strategy for the proof is the same, we only indicate modification of our arguments. Among them, we especially treat with a Liouville-type theorem and a blow-up argument, since these proofs require basic facts of a general convex functional which is positively homogeneous of degree $1$. These well-known facts are contained in the appendices for completeness.

\section{Regularity outside the facet}\label{Sect Regularity outside facets}
In Section \ref{Sect Regularity outside facets}, we would like to show that \(u\) is \(C^{1}\) at any \(x\in D\), and therefore (\ref{reduced question}) holds for all \(x\in D\). This result will be used in the proof of the strong maximum principle (Theorem \ref{strong maximum principle}).

We first give a precise definition of weak solutions to \(L_{b,\,p}u=f\) in a convex domain \(\Omega\subset{\mathbb R}^{n}\), which is not necessarily bounded.
\begin{definition}\upshape\label{def of weak solution in total space}\upshape
Let \(\Omega\subset{\mathbb R}^{n}\) be a domain, which is not necessarily bounded, and \(f\in L_{\mathrm{loc}}^{q}(\Omega)\,(n<q\le\infty)\).
We say that a function \(u\in W_{\mathrm{loc}}^{1,\,p}(\Omega)\) is a \textit{weak} solution to (\ref{crystal model eq again}), when for any bounded Lipschitz domain \(\omega\Subset\Omega\), there exists a vector field \(Z\in L^{\infty}(\omega,\,{\mathbb R}^{n})\) such that the pair \((u,\,Z)\in W^{1,\,p}(\omega)\times L^{\infty}(\omega,\,{\mathbb R}^{n})\) satisfies \begin{equation}\label{weak formulation}
b\int_{\omega}\langle Z\mid\nabla\phi\rangle\,dx+\int_{\omega}\mleft\langle\lvert\nabla u\rvert^{p-2}\nabla u \mathrel{}\middle|\mathrel{} \nabla\phi\mright\rangle\,dx=\int_{\omega}f\phi\,dx 
\end{equation}
for all \(\phi\in W_{0}^{1,\,p}(\omega)\), and
\begin{equation}\label{subgradient Z}
Z(x)\in\partial\lvert\,\cdot\,\rvert(\nabla u(x))
\end{equation}
for a.e. \(x\in\omega\). For such pair \((u,\,Z)\), we say that $(u,\,Z)$ satisfies $L_{b,\,p}u=f$ in $W^{-1,\,p^{\prime}}(\omega)$ or simply say that $u$ satisfies $L_{b,\,p}u=f$ in $W^{-1,\,p^{\prime}}(\omega)$. Here \(p^{\prime}\in(1,\,\infty)\) denotes the H\"{o}lder conjugate exponent of \(p\in(1,\,\infty)\).
\end{definition}

The aim of Section \ref{Sect Regularity outside facets} is to show Lemma \ref{regularity at regular point} below.
\begin{lemma}\label{regularity at regular point}
Let \(u\) be a convex weak solution to (\ref{crystal model eq again}) in a convex domain \(\Omega\subset{\mathbb R}^{n}\), and \(f\in L^{q}_{\mathrm{loc}}(\Omega)\,(n<q\le\infty)\). If \(x_{0}\in D\), then we can take a small radius \(r_{0}>0\), a unit vector \(\nu_{0}\in{\mathbb R}^{n}\), and a small number \(\mu_{0}>0\) such that
\begin{equation}\label{essinf>0}
\overline{B_{r_{0}}(x_{0})}\subset D\textrm{ and }\langle\nabla u(x)\mid \nu_{0}\rangle\ge \mu_{0}\quad\textrm{for a.e. }x\in B_{r_{0}}(x_{0}),
\end{equation}
and there exists a small number \(\alpha=\alpha(\mu_{0})\in (0,\,1)\) such that \(u\in C^{1,\,\alpha}(B_{r_{0}/2}(x_{0}))\). In particular, \(u\) is \(C^{1}\) in \(D\), and \(\partial u(x)=\{\nabla u(x)\}\not= \{0\}\) for all \(x\in D\).
\end{lemma}
Before proving Lemma \ref{regularity at regular point}, we introduce difference quotients.
For given \(g\colon \Omega\rightarrow {\mathbb R}^{m}\,(m\in{\mathbb N}),\,j\in\{\,1,\,\dots,\,n\,\},\,h\in{\mathbb R}\setminus\{0\}\), we define 
\[\Delta_{j,\,h}g(x)\coloneqq \frac{g(x+he_{j})-g(x)}{h}\in {\mathbb R}^{m}\quad\textrm{for \(x\in \Omega\) with }x+he_{j}\in \Omega,\]
where \(e_{j}\in{\mathbb R}^{n}\) denotes the unit vector in the direction of the \(x_{j}\)-axis.

In the proof of Lemma \ref{regularity at regular point}, we will use Lemma \ref{Poincare-type lemma for differential quotient}--\ref{coercivity lemma for convex function} without proofs. For precise proofs, see Section \ref{Appendix A}.
\begin{proof}
For each fixed \(x_{0}\in D\), we may take and fix \(x_{1}\in \Omega\) such that \(u(x_{0})>u(x_{1})\).
We set \(3\delta_{0}\coloneqq u(x_{0})-u(x_{1})>0,\,d_{0}\coloneqq \lvert x_{0}-x_{1}\rvert>0\) and \(\nu_{0}\coloneqq d_{0}^{-1}(x_{0}-x_{1})\). By \(u\in C(\Omega)\), we may take a sufficiently small \(r_{0}>0\) such that 
\begin{equation}\label{separation by balls}
u(y_{0})-u(y_{1})\ge \delta_{0}>0\quad \textrm{for all}\quad y_{0}\in B_{r_{0}}(x_{0}),\,y_{1}\in B_{r_{0}}(x_{1}).
\end{equation}
From (\ref{separation by balls}), the inclusion \(\overline{B_{r_{0}}(x_{0})}\subset D\) clearly holds.
(\ref{separation by balls}) also allows us to check that for all \(y_{0}\in B_{r_{0}}(x_{0}),\,z_{0}\in\partial u(y_{0})\),
\begin{equation}\label{essinf;subgradient ver}
\langle z_{0}\mid \nu_{0}\rangle \ge \frac{u(y_{0})-u(y_{0}-d_{0}\nu_{0})}{d_{0}}\ge \frac{\delta_{0}}{d_{0}}\eqqcolon \mu_{0}>0.
\end{equation}
For the first inequality in (\ref{essinf;subgradient ver}), we have used Lemma \ref{coercivity lemma for convex function}, which is basically derived from convexity of \(u\).
Recall that \(\partial u(x)=\{\nabla u(x)\}\) for a.e. \(x\in\Omega\), and hence we are able to recover (\ref{essinf>0}) from (\ref{essinf;subgradient ver}). 

In order to obtain \(C^{1}\)-regularity in $D$, we will appeal to the classical De Giorgi--Nash--Moser theory.
For preliminaries, we check that the operator \(L_{b,\,p}u\) assures uniform ellipticity in \(B_{r_{0}}(x_{0})\).
Local Lipschitz continuity of $u$ implies that there exists a sufficiently large number \(M_{0}\in(0,\,\infty)\) such that
\begin{equation}\label{lipschitz bound M}
\esssup\limits_{B_{r_{0}}(x_{0})}\,\lvert\nabla u\rvert\le M_{0}\textrm{ and }\lvert u(x)-u(y)\rvert\le M_{0}\lvert x-y\rvert\quad\textrm{for all }x,\,y\in B_{r_{0}}(x_{0}).
\end{equation}
For notational simplicity, we write subdomains by \[U_{1}\coloneqq B_{r_{0}}(x_{0})\Supset U_{2}\coloneqq B_{15r_{0}/16}(x_{0})\Supset U_{3}\coloneqq B_{7r_{0}/8}(x_{0})\Supset U_{4}\coloneqq B_{3r_{0}/4}(x_{0})\Supset U_{5}\coloneqq B_{r_{0}/2}(x_{0}).\]
It should be noted that \(E(z)\coloneqq b\lvert z\rvert+\lvert z\rvert^{p}/p\,(z\in{\mathbb R}^{n})\) satisfies \(E\in C^{\infty}({\mathbb R}^{n}\setminus\{0\})\), and there exists two constants \(0<\lambda(p,\,\mu_{0},\,M_{0})\le \Lambda(b,\,p,\,\mu_{0},\,M_{0})<\infty\) such that
\begin{equation}\label{uniformly elliptic on a ball in D}
\lambda \lvert \zeta\rvert^{2}\le \mleft\langle \nabla_{z}^{2}E(z_{0})\zeta\mathrel{}\middle|\mathrel{}\zeta\mright\rangle
\end{equation}
\begin{equation}\label{uniformly bound on a ball in D}
\mleft\langle\nabla_{z}^{2}E(z_{0})\zeta\mathrel{}\middle|\mathrel{} \omega\mright\rangle\le \Lambda\lvert\zeta\rvert\lvert\omega\rvert
\end{equation}
for all \(z_{0},\,\zeta,\,\omega\in{\mathbb R}^{n}\) with \(\mu_{0}\le \lvert z_{0}\rvert\le M_{0}\).
We can explicitly determine \(0<\lambda\le \Lambda<\infty\) by
\[\mleft\{\begin{array}{ccc}
\lambda(p,\,\mu_{0},\,M_{0})&\coloneqq& \min_{\mu_{0}\le t\le M_{0}}\mleft(\min\{\,1,\,p-1\,\}t^{p-2}\mright),\\ \Lambda(b,\,p,\,\mu_{0},\,M_{0})&\coloneqq&\max_{\mu_{0}\le t\le M_{0}}\mleft(bt^{-1}+\max\{\,1,\,p-1\,\}t^{p-2}\mright)
\end{array} \mright.\]

Now we check that \(u\in W^{2,\,2}(U_{4})\) by the difference quotient method. We refer the reader to \cite[Theorem 8.1]{MR1962933} as a related result. 
By \cite[Lemma 8.2]{MR1962933}, it suffices to check that
\begin{equation}\label{uniform estimate of differential quotient}
\sup\,\mleft\{\int_{U_{4}}\lvert \nabla(\Delta_{j,\,h}u)\rvert^{2}\,dx\mathrel{}\middle|\mathrel{}h\in{\mathbb R},\,0<\lvert h\rvert<\frac{r_{0}}{16}\mright\}<\infty\quad\textrm{for each }j\in\{\,1,\,\dots,\,n\,\}.
\end{equation}
Since \(u\in W^{1,\,p}(U_{1})\) satisfies \(L_{b,\,p}u=f\) in \(W^{-1,\,p^{\prime}}(U_{1})\), we obtain
\begin{equation}\label{weak form over a closed ball in D}
\int_{U_{1}}\mleft\langle\nabla_{z}E(\nabla u) \mathrel{}\middle|\mathrel{}\nabla\phi\mright\rangle\,dx=\int_{U_{1}}f\phi\,dx
\end{equation}
for all \(\phi\in W_{0}^{1,\,p}(U_{1})\).
Here we note that \(\nabla u\) no longer degenerates in \(U_{1}\) by (\ref{essinf>0}).
We fix a cutoff function \(\eta\in C_{c}^{1}(U_{3})\) such that
\begin{equation}\label{cutoff}
0\le\eta\le 1\textrm{ in }U_{3},\,\eta\equiv 1\textrm{ in }U_{4},\, \lvert\nabla\eta\rvert\le \frac{c}{r_{0}}
\end{equation}
for some constant \(c>0\). 
For each fixed \(j\in\{\,1,\,\dots\,,\,n\,\},\,h\in{\mathbb R}\) with \(0<\lvert h\rvert<r_{0}/16\), we test \(\phi\coloneqq \Delta_{j,\,-h}(\eta^{2}\Delta_{j,\,h}u)\) into (\ref{weak form over a closed ball in D}).
We note that \(\phi\in W^{1,\,\infty}(U_{1})\subset W^{1,\,p}(U_{1})\) by (\ref{uniformly bound on a ball in D}), and this is compactly supported in \(U_{2}\). Hence \(\phi\in W_{0}^{1,\,p}(U_{2})\) is an admissible test function.
By testing \(\phi\), we have
\begin{align}\label{tested equation for differential quotient}
0&=\int_{U_{2}}\mleft\langle \Delta_{j,\,h}(\nabla_{z}E(\nabla u(x)))\mathrel{}\middle|\mathrel{}\eta^{2}\nabla(\Delta_{j,\,h}u)+2\eta\Delta_{j,\,h}u\nabla\eta\mright\rangle -\int_{U_{2}}f\Delta_{-j,\,h}(\eta^{2}\Delta_{j,\,h}u)\,dx\nonumber\\&=\int_{U_{2}}\eta^{2}\mleft\langle A_{h}(x,\,\nabla u(x))\nabla(\Delta_{j,\,h}u)\mathrel{}\middle|\mathrel{}\nabla(\Delta_{j,\,h}u)\mright\rangle\,dx\nonumber\\&\quad+2\int_{U_{2}}\eta\Delta_{j,\,h}u\mleft\langle A_{h}(x,\,\nabla u(x))\nabla(\Delta_{j,\,h}u)\mathrel{}\middle|\mathrel{}\nabla\eta\mright\rangle\,dx\nonumber\\&\quad\quad -\int_{U_{2}}f\Delta_{-j,\,h}(\eta^{2}\Delta_{j,\,h}u)\,dx\nonumber\\&\eqqcolon I_{1}+I_{2}-I_{3}.
\end{align}
Here \(A_{h}=A_{h}(x,\,\nabla u(x))\) denotes a matrix-valued function in $U_{2}$ given by 
\[A_{h}(x,\,\nabla u(x))\coloneqq\int_{0}^{1}\nabla_{z}^{2}E((1-t)\nabla u(x)+t\nabla u(x+he_{j}))\,dt.\]
We note that with the aid of (\ref{essinf>0})--(\ref{lipschitz bound M}), we obtain
\[\mu_{0}\le \lvert(1-t)\nabla u(x)+t\nabla u(x+he_{j})\lvert\le M_{0}\]
for a.e. \(x\in U_{2}\) and for all \(0\le t\le 1\).
Combining this result with (\ref{uniformly elliptic on a ball in D})--(\ref{uniformly bound on a ball in D}), we conclude that \(A_{h}\) satisfies
\begin{equation}\label{uniformly elliptic for Ah}
\lambda \lvert \zeta\rvert^{2}\le \mleft\langle A_{h}(x,\,\nabla u(x))\zeta\mathrel{}\middle|\mathrel{}\zeta\mright\rangle
\end{equation}
\begin{equation}\label{uniformly bound for Ah}
\mleft\langle A_{h}(x,\,\nabla u(x))\zeta\mathrel{}\middle|\mathrel{} \omega\mright\rangle\le \Lambda\lvert\zeta\rvert\lvert\omega\rvert
\end{equation}
for all \(\zeta,\,\omega\in{\mathbb R}^{n}\) and for a.e. \(x\in U_{2}\).
We set an integral
\[J\coloneqq \int_{U_{2}}\eta^{2}\lvert\nabla(\Delta_{j,\,h}u)\rvert^{2}\,dx.\]
By (\ref{uniformly elliptic for Ah}), it is clear that \(I_{1}\ge\lambda J\).
By Young's inequality and applying a Poincar\'{e}-type inequality (Lemma \ref{Poincare-type lemma for differential quotient}) to \(\eta^{2}\Delta_{j,\,h}u\in W_{0}^{1,\,2}(U_{2})\), we obtain for any \(\varepsilon>0\),
\begin{align*}
\lvert I_{3}\rvert&\le \frac{1}{4\varepsilon}\lVert f\rVert_{L^{2}(U_{2})}^{2}+\varepsilon\int_{U_{2}}\lvert \nabla (\eta^{2}\Delta_{j,\,h}u)\rvert^{2}\,dx\\&\le \frac{1}{4\varepsilon}\lVert f\rVert_{L^{2}(U_{2})}^{2}+4\varepsilon\int_{U_{2}}\lvert \Delta_{j,\,h}u\rvert^{2}\lvert\nabla\eta\rvert^{2}\,dx+2\varepsilon\int_{U_{2}}\eta^{2}\lvert\nabla(\Delta_{j,\,h}u)\rvert^{2}\,dx.
\end{align*}
Here we have invoked the property \(0\le \eta\le 1\) in \(U_{2}\).
We fix \(\varepsilon\coloneqq \lambda/6>0\).
By (\ref{uniformly bound for Ah}) and Young's inequality, we have
\begin{align*}
\lvert I_{2}\rvert&\le 2\Lambda\int_{U_{2}}\eta \lvert \nabla(\Delta_{j,\,h}u)\rvert\cdot\lvert \Delta_{j,\,h}u\rvert\lvert\nabla\eta\rvert\,dx\\&\le \frac{\lambda}{3}J+\frac{3\Lambda^{2}}{\lambda}\int_{U_{2}}\lvert \Delta_{j,\,h}u\rvert^{2}\lvert\nabla\eta\rvert^{2}\,dx.
\end{align*}
It follows from (\ref{lipschitz bound M}) that \(\lVert \Delta_{j,\,h}u\rVert_{L^{\infty}(U_{2})}\le M_{0}\). Therefore we obtain from (\ref{tested equation for differential quotient}),
\[\int_{U_{4}}\lvert \nabla(\Delta_{j,\,h}u)\rvert^{2}\,dx\le J=\int_{U_{2}}\eta^{2}\lvert\nabla(\Delta_{j,\,h}u)\rvert^{2}\,dx\le C(\lambda,\,\Lambda)\mleft(M_{0}^{2}\lVert\nabla\eta\rVert_{L^{2}(U_{2})}^{2}+\lVert f\rVert_{L^{2}(U_{2})}^{2}\mright).\]
The estimate (\ref{uniform estimate of differential quotient}) follows from this,  and therefore \(u\in W^{2,\,2}(U_{4})\).

For each \(\psi\in C_{c}^{\infty}(U_{4})\), we test \(\partial_{x_{j}}\psi\in C_{c}^{\infty}(U_{4})\) into (\ref{weak form over a closed ball in D}). Integrating by parts, we obtain
\begin{equation}\label{weak form over a closed ball in D differentiated}
\int_{U_{4}}\mleft\langle\nabla_{z}^{2}E(\nabla u)\nabla \partial_{x_{j}}u \mathrel{}\middle|\mathrel{}\nabla\psi\mright\rangle\,dx=-\int_{U_{4}}f\partial_{x_{j}}\psi\,dx
\end{equation}
for all \(\psi\in C_{c}^{\infty}(U_{4})\). Noting that \(f\in L^{q}(U_{4})\subset L^{2}(U_{4}),\,\partial_{x_{j}}u\in W^{1,\,2}(U_{4})\), and (\ref{uniformly elliptic on a ball in D})--(\ref{uniformly bound on a ball in D}), we may extend \(\psi\in W_{0}^{1,\,2}(U_{4})\) by a density argument.
The conditions (\ref{uniformly elliptic on a ball in D})--(\ref{uniformly bound on a ball in D}) imply that \(\nabla_{z}^{2}E(\nabla u)\) is uniformly elliptic over \(U_{1}\).
Hence by \cite[Theorem 8.22]{MR1814364}, there exists \(\alpha=\alpha(\lambda,\,\Lambda,\,n,\,q)\in (0,\,1)\) such that \(\partial_{x_{j}}u\in C^{\alpha}(U_{5})\) for each \(j\in\{\,1,\,\dots,\,n\,\}\). This regularity result implies \(\partial u(x)=\{\nabla u(x)\}\not= \{0\}\) for all \(x\in D\).
\end{proof}

\section{A blow-up argument}\label{Sect Blow-argument}
In order to show that (\ref{reduced question}) holds true even for \(x\in F\), we first make a blow-argument and construct a convex weak solution in the whole space \({\mathbb R}^{n}\), in the sense of Definition \ref{def of weak solution in total space}.
\begin{proposition}\label{limit of rescaled solutions}
Let \(\Omega\subset{\mathbb R}^{n}\) be a convex domain, and \(f\in L^{q}_{\mathrm{loc}}(\Omega)\,(n<q\le\infty)\). Assume that \(u\) is a convex weak solution to (\ref{crystal model eq again}), and \(x_{0}\in\Omega\). Then there exists a convex function \(u_{0}\colon {\mathbb R}^{n}\rightarrow {\mathbb R}\) such that
\begin{enumerate}
\item \(u_{0}\) is a weak solution to \(L_{b,\,p}u_{0}=0\) in \({\mathbb R}^{n}\).
\item The inclusion \(\partial u(x_{0})\subset\partial u_{0}(x_{0})\) holds. That is, if \(c\in\partial u(x_{0})\), then we have
\[u_{0}(x)\ge u_{0}(x_{0})+\langle c\mid x-x_{0}\rangle\quad\textrm{for all }x\in{\mathbb R}^{n}.\]
In particular, if \(x_{0}\in F\), then the facet of \(u_{0}\) is non-empty.
\end{enumerate}
\end{proposition}
\begin{proof}
Without loss of generality, we may assume that \(x_{0}=0\) and \(u(x_{0})=0\).
First we fix a closed ball \(\overline{B_{R}(0)}=\overline{B_{R}}\subset \Omega\). We note that \(u\in \Lip (\overline{B_{R}})\) since \(u\) is convex. Hence there exists a sufficiently large number \(M\in(0,\,\infty)\) such that
\[\esssup\limits_{B_{R}}\,\lvert\nabla u\rvert\le M\textrm{ and }\lvert u(x)-u(y)\rvert\le M\lvert x-y\rvert\quad\textrm{for all }x,\,y\in B_{R}.\]
We take and fix a vector field \(Z\in L^{\infty}(B_{R},\,{\mathbb R}^{n})\) such that the pair \((u,\,Z)\in W^{1,\,p}(B_{R})\times L^{\infty}(B_{R},\,{\mathbb R}^{n})\) satisfies \(L_{b,\,p}u=f\) in \(W^{-1,\,p^{\prime}}(B_{R})\).
For each \(a>0\), we define a rescaled convex function \(u_{a}\colon B_{R/a}\rightarrow {\mathbb R}\) and a dilated vector field \(Z_{a}\in L^{\infty}(B_{R/a},\,{\mathbb R}^{n})\) by
\[u_{a}(x)\coloneqq \frac{u(ax)}{a},\quad Z_{a}(x)\coloneqq Z(ax)\quad\textrm{for }\,x\in B_{R/a}.\]
We also set \(f_{a}\in L^{q}(B_{R/a})\) by \[f_{a}(x)\coloneqq af(ax)\quad \textrm{for }\,x\in B_{R/a}.\]
Then it is easy to check that the pair \((u_{a},\,Z_{a})\in W^{1,\,\infty}(B_{R/a})\times L^{\infty}(B_{R/a},\,{\mathbb R}^{n})\) satisfies \(L_{b,\,p}u_{a}=f_{a}\) in \(W^{-1,\,p^{\prime}}(B_{R/a})\).
For each fixed $R<r<\infty$, the inclusion \(B_{r}=B_{r}(0)\subset B_{R/a}\) holds for all $a\in(0,\,R/r)$.
We also have
\begin{equation}\label{bounds of rescaled functions ua}
\sup\limits_{B_{r}}\,\lvert u_{a}\rvert\le \frac{r}{R}M<\infty,\quad \lVert \nabla u_{a}\rVert_{L^{\infty}(B_{r})}\le M<\infty\quad \textrm{for all }a\in(0,\,R/r)\end{equation}
by definition of $u_{a}$.
Hence by the Arzel\`{a}--Ascoli theorem and a diagonal argument, we can take a decreasing sequence \(\{a_{N}\}_{N=1}^{\infty}\subset(0,\,\infty)\), such that \(a_{N}\to 0\) as \(N\to\infty\), and
\begin{equation}\label{compact convergence}
u_{a_{N}}\rightarrow u_{0}\quad\textrm{locally uniformly in }{\mathbb R}^{n}.
\end{equation}
for some function \(u_{0}\colon{\mathbb R}^{n}\to{\mathbb R}\).
Clearly \(u_{0}\) is convex in \({\mathbb R}^{n}\), and the inclusion \(\partial u(x_{0})\subset\partial u_{0}(x_{0})\) holds true by the construction of rescaled functions \(u_{a}\).
If \(x_{0}\in F\), then we have \(\{0\}\subset\partial u(x_{0})\subset\partial u_{0}(x_{0})\) and therefore \(x_{0}\) lies in the facet of \(u_{0}\).
We are left to show that \(u_{0}\) is a weak solution to \(L_{b,\,p}u_{0}=0\) in \({\mathbb R}^{n}\).
Before proving this, we note that from (\ref{bounds of rescaled functions ua})--(\ref{compact convergence}) and Lemma \ref{a.e. pointwise convergence for convex functions}, it follows that
\begin{equation}\label{a.e. convergence for rescaled vector}
\nabla u_{a_{N}}(x)\rightarrow \nabla u_{0}(x)\quad\textrm{and}\quad\lvert \nabla u_{0}(x)\rvert\le M\quad\textrm{for a.e. }x\in{\mathbb R}^{n}
\end{equation}
as \(N\to\infty\).
We arbitrarily fix an open ball \(B_{r}=B_{r}(0)\subset{\mathbb R}^{n}\).
We easily realize that a family of pairs \(\{(u_{a},\,Z_{a})\}_{0<a<R/r}\subset W^{1,\,\infty}(B_{r})\times L^{\infty}(B_{r},\,{\mathbb R}^{n})\) satisfies 
\begin{equation}\label{weak condition}
Z_{a}(x)\in\partial\lvert\,\cdot\,\rvert(\nabla u_{a}(x))\quad \textrm{for a.e. }x\in B_{r},
\end{equation}
\begin{equation}\label{weak formulation in local domain}
b\int_{B_{r}}\langle Z_{a}\mid \nabla \phi\rangle\,dx+\int_{B_{r}}\mleft\langle \lvert\nabla u_{a}\rvert^{p-2}\nabla u_{a}\mathrel{}\middle|\mathrel{}\nabla\phi \mright\rangle\,dx=\int_{B_{r}}f_{a}\phi\,dx\quad\textrm{for all }\phi\in W_{0}^{1,\,p}(B_{r}).
\end{equation}
By definition of \(f_{a}\), we get \(\lVert f_{a}\rVert_{L^{q}(B_{r})}=a^{1-n/q}\lVert f\rVert_{L^{q}(B_{ar})}\le a^{1-n/q}\lVert f\rVert_{L^{q}(B_{R})}\) for all \(0<a<R/r\). Hence by the continuous embedding \(L^{q}(B_{r})\hookrightarrow W^{-1,\,p^{\prime}}(B_{r})\), we obtain
\begin{equation}\label{strong converge to 0}
f_{a_{N}}\rightarrow 0\quad\textrm{ in }W^{-1,\,p^{\prime}}(B_{r})\quad\textrm{as }N\to \infty.
\end{equation}
By (\ref{bounds of rescaled functions ua}) and (\ref{a.e. convergence for rescaled vector}), we can apply Lebesgue's dominated convergence theorem and get
\begin{equation}\label{convergence in Lp' for rescaled vector field}
\lvert \nabla u_{a_{N}}\rvert^{p-2}\nabla u_{a_{N}}\rightarrow \lvert\nabla u_{0}\rvert^{p-2}\nabla u_{0}\quad\textrm{ in }L^{p^{\prime}}(B_{r},\,{\mathbb R}^{n})\quad\textrm{as }N\to \infty.
\end{equation}
It is clear that \(\lVert Z_{a}\rVert_{L^{\infty}(B_{r},\,{\mathbb R}^{n})}\le 1\) for all \(0<a<R/r\). Hence by \cite[Corollary 3.30]{MR2759829}, up to a subsequence, we may assume that
\begin{equation}\label{weakstar convergence}
Z_{a_{N}} \overset{\ast}{\rightharpoonup} Z_{0,\,r}\quad \textrm{ in }L^{\infty}\mleft(B_{r},\,{\mathbb R}^{n}\mright)\quad\textrm{as }N\to \infty
\end{equation}
for some \(Z_{0,\,r}\in L^{\infty}(B_{r},\,{\mathbb R}^{n})\). By lower-semicontinuity of the norm with respect to the weak\(^{\ast}\) topology and (\ref{a.e. convergence for rescaled vector})--(\ref{weak condition}), we get
\[\lVert Z_{0,\,r}\rVert_{L^{\infty}(B_{r},\,{\mathbb R}^{n})}\le 1,\quad Z_{0,\,r}(x)=\frac{\nabla u_{0}(x)}{\lvert \nabla u_{0}(x)\rvert}\quad\textrm{for a.e. }x\in B_{r} \textrm{ with }\nabla u_{0}(x)\not= 0,\]
which implies that
\begin{equation}\label{z0r is subgradient vector field}
Z_{0,\,r}(x)\in \partial \lvert\,\cdot\,\rvert(\nabla u_{0}(x))\quad\textrm{for a.e. }x\in B_{r}.
\end{equation}
Letting \(a=a_{N}\) in (\ref{weak formulation in local domain}) and \(N\to \infty\), we obtain
\begin{equation}\label{limit of weak formulaton}
b\int_{B_{r}}\langle Z_{0,\,r}\mid \nabla \phi\rangle\,dx+\int_{B_{r}}\mleft\langle \lvert\nabla u_{0}\rvert^{p-2}\nabla u_{0}\mathrel{}\middle|\mathrel{}\nabla\phi \mright\rangle\,dx=0\quad\textrm{for all }\phi\in W_{0}^{1,\,p}(B_{r})
\end{equation}
by (\ref{weak formulation in local domain})--(\ref{weakstar convergence}).
Since \(B_{r}\subset{\mathbb R}^{n}\) is arbitrary, (\ref{z0r is subgradient vector field})--(\ref{limit of weak formulaton}) means that \(u_{0}\) is a weak solution to \(L_{b,\,p}u_{0}=0\) in \({\mathbb R}^{n}\), in the sense of Definition \ref{def of weak solution in total space}.
\end{proof}

\section{Maximum principles}\label{Sect Maximum Principles}
In Section \ref{Sect Maximum Principles}, we justify maximum principles for the equation \(L_{b,\,p}u=0\).

We first define subsolutions and supersolutions in the weak sense.
\begin{definition}\upshape\label{def of weak sub/supersol}\upshape
Let \(\Omega\subset{\mathbb R}^{n}\) be a bounded domain. 
A pair \((u,\,Z)\in W^{1,\,p}(\Omega)\times L^{\infty}(\Omega,\,{\mathbb R}^{n})\) is called a \textit{weak} subsolution to \(L_{b,\,p}u=0\) in \(\Omega\), if it satisfies
\begin{equation}\label{weak subsolution}
b\int_{\Omega}\langle Z \mid \nabla\phi\rangle\,dx+\int_{\Omega}\mleft\langle \lvert\nabla u\rvert^{p-2}\nabla u\mathrel{}\middle|\mathrel{}\nabla\phi\mright\rangle\,dx\le 0\end{equation}
for all $0\le\phi\in C_{c}^{\infty}(\Omega)$, and
\begin{equation}\label{gradient vector field}
Z(x)\in\partial\lvert\,\cdot \,\rvert\mleft(\nabla u(x)\mright)\quad\textrm{for a.e. }x\in\Omega.
\end{equation}
Similarly we call a pair \((u,\,Z)\in W^{1,\,p}(\Omega)\times L^{\infty}(\Omega,\,{\mathbb R}^{n})\) a \textit{weak} supersolution \(L_{b,\,p}u=0\) in \(\Omega\), if it satisfies (\ref{gradient vector field}) and 
\begin{equation}\label{weak supersolution}
b\int_{\Omega}\langle Z \mid \nabla\phi\rangle\,dx+\int_{\Omega}\mleft\langle \lvert\nabla u\rvert^{p-2}\nabla u\mathrel{}\middle|\mathrel{}\nabla\phi\mright\rangle\,dx\ge 0
\end{equation}
for all $0\le\phi\in C_{c}^{\infty}(\Omega)$.
For \(u\in W^{1,\,p}(\Omega)\), we simply say that \(u\) is respectively a subsolution and a supersolution to \(L_{b,\,p}u=0\) in the \textit{weak} sense if there is \(Z\in L^{\infty}(\Omega,\,{\mathbb R}^{n})\) such that the pair \((u,\,Z)\) is a weak subsolution and a weak supersolution to \(L_{b,\,p}u=0\) in \(\Omega\).
\end{definition}
\begin{remark}\upshape\label{some remarks on solutions}\upshape
We describe some remarks on our definitions of weak solutions, subsolutions and supersolutions. 
\begin{enumerate}
\item \label{extension remark} 
By an approximation argument, we may extend the test function class of (\ref{weak subsolution}) to
\[D_{+}(\Omega)\coloneqq\{\phi\in W^{1,\,p}(\Omega)\mid \phi\ge 0\textrm{ a.e. in }\Omega,\,\, \supp \phi\subset \Omega\}.\]
Indeed, for \(\phi\in D_{+}(\Omega)\) and \(0<\varepsilon<\dist(\supp\phi,\,\partial\Omega)\), the function,
\[\phi_{\varepsilon}(x)=\int_{\Omega}\phi(x-y)\rho_{\varepsilon}(y)\,dy\quad\textrm{for }x\in\Omega\]
satisfies \(0\le\phi_{\varepsilon}\in C_{c}^{\infty}(\Omega)\).
Here for \(0<\varepsilon<\infty\), \(0\le \rho_{\varepsilon}\in C_{c}^{\infty}(B_{\varepsilon}(0))\) denotes a standard mollifier so that
\[0\le \rho\in C_{c}^{\infty}(B_{1}),\quad \lVert \rho\rVert_{L^{1}({\mathbb R}^{n})}=1,\quad\rho_{\varepsilon}(x)\coloneqq \varepsilon^{-n}\rho(x/\varepsilon)\textrm{ for }x\in{\mathbb R}^{n}.\]
By testing \(\phi_{\varepsilon}\) into (\ref{weak subsolution}) for sufficiently small \(\varepsilon>0\) and letting \(\varepsilon\to 0\), we conclude that if the pair \((u,\,Z)\) satisfies (\ref{weak subsolution}) for all \(0\le\phi\in C_{c}^{\infty}(\Omega)\), then (\ref{weak subsolution}) holds for all \(\phi\in D_{+}(\Omega)\). A similar result is also valid for (\ref{weak supersolution}).
\item By Definition \ref{def of weak solution in total space}--\ref{def of weak sub/supersol}, if a pair \((u,\,Z)\in W^{1,\,p}(\Omega)\times L^{\infty}(\Omega,\,{\mathbb R}^{n})\) satisfies \(L_{b,\,p}u=0\) in \(W^{-1,\,p^{\prime}}(\Omega)\), then \(u\) is clearly both a subsolution and a supersolution to \(L_{b,\,p}u=0\) in \(\Omega\) in the weak sense. Conversely, if a pair \((u,\,Z)\in W^{1,\,p}(\Omega)\times L^{\infty}(\Omega,\,{\mathbb R}^{n})\) is both a weak subsolution and a weak supersolution to \(L_{b,\,p}u=0\) in \(\Omega\), then the pair \((u,\,Z)\) satisfies \(L_{b,\,p}u=0\) in \(W^{-1,\,p^{\prime}}(\Omega)\). Indeed, by the previous remark we have already known that the pair $(u,\,Z)$ satisfies (\ref{weak subsolution}) and (\ref{weak supersolution}) for all \(\phi\in D_{+}(\Omega)\), which clearly yields
\begin{equation}\label{weak formulation in remark}
b\int_{\Omega}\langle Z\mid\nabla\phi\rangle\,dx+\int_{\Omega}\mleft\langle\lvert\nabla u\rvert^{p-2}\nabla u \mathrel{}\middle|\mathrel{} \nabla\phi\mright\rangle\,dx=0
\end{equation}
for all \(\phi\in D_{+}(\Omega)\).
We decompose arbitrary \(\phi\in C_{c}^{\infty}(\Omega)\) by \(\phi=\phi_{+}-\phi_{-}\), where \(\phi_{+}\coloneqq \max\{\,\phi,\,0\,\},\,\phi_{-}\coloneqq \max\{\,-\phi,\,0\,\}\in D_{+}(\Omega)\). By testing \(\phi_{+},\,\phi_{-}\in D_{+}(\Omega)\) into (\ref{weak formulation in remark}), we conclude that (\ref{weak formulation in remark}) holds for all \(\phi\in C_{c}^{\infty}(\Omega)\). By density of \(C_{c}^{\infty}(\Omega)\subset W_{0}^{1,\,p}(\Omega)\), it is clear that (\ref{weak formulation in remark}) is valid for all \(\phi\in W_{0}^{1,\,p}(\Omega)\). 
\item For a bounded domain \(\Omega\subset{\mathbb R}^{n}\), let \(u\in C^{2}(\overline{\Omega})\) satisfy the following two conditions (\ref{nonvanishing})--(\ref{subsolution in the classical sense});
\begin{equation}\label{nonvanishing}
\nabla u(x) \not= 0\quad \textrm{for all }x\in\Omega,
\end{equation}
\begin{equation}\label{subsolution in the classical sense}
(L_{b,\,p}u)(x)=-(b\Delta_{1}u+\Delta_{p}u)(x)\le 0 \quad \textrm{for all }x\in\Omega.
\end{equation}
Then for any fixed \(0\le \phi\in C_{c}^{\infty}(\Omega)\), we have
\[0\ge \int_{\Omega}(L_{b,\,p}u)\phi\,dx=b\int_{\Omega}\mleft\langle\frac{\nabla u}{\lvert\nabla u\rvert}\mathrel{}\middle|\mathrel{}\nabla\phi\mright\rangle\,dx+\int_{\Omega}\mleft\langle \lvert\nabla u\rvert^{p-2}\nabla u\mathrel{}\middle|\mathrel{}\nabla\phi\mright\rangle\,dx,\]
with the aid of integration by parts and (\ref{subsolution in the classical sense}).
We also note that 
\[\partial\lvert\,\cdot\,\rvert(\nabla u(x))=\mleft\{\frac{\nabla u(x)}{\lvert \nabla u(x)\rvert}\mright\}\quad \textrm{for all }x\in\Omega\]
by (\ref{nonvanishing}).
Therefore the pair \((u,\,\nabla u/\lvert\nabla u\rvert)\in W^{1,\,p}(\Omega)\times L^{\infty}(\Omega,\,{\mathbb R}^{n})\) satisfies (\ref{weak subsolution})--(\ref{gradient vector field}).
For such \(u\), we simply say that \(u\) satisfies \(L_{b,\,p}u\le 0\) in \(\Omega\) in the \textit{classical} sense.
\end{enumerate}
\end{remark}
\subsection{Comparison principle}\label{Subsect Comparison Principle}
We justify the comparison principle, i.e., for any subsolution \(u^{-}\) and supersolution \(u^{+}\),
\[u^{-}\le u^{+}\quad\textrm{on }\partial\Omega\quad \textrm{implies that}\quad u^{-}\le u^{+}\quad \textrm{in }\Omega,\]
under the condition that \(u^{+}\) and \(u^{-}\) admits continuity properties in \(\overline{\Omega}\). 
\begin{proposition}\label{comparison principle}
Let \(\Omega\subset{\mathbb R}^{n}\) be a bounded domain.
Assume that \(u^{+},\,u^{-}\in C(\overline{\Omega})\cap W^{1,\,p}(\Omega)\) is a subsolution and a supersolution to \(L_{b,\,p}u=0\) in the weak sense respectively. If \(u^{+},\,u^{-}\) satisfies
\begin{equation}\label{pointwise ineqality}
u^{-}(x)\le u^{+}(x)\quad \textrm{for all }x\in \partial\Omega,
\end{equation}
then \(u^{-}\le u^{+}\) in \(\overline{\Omega}\).
\end{proposition}
Before proving Proposition \ref{comparison principle}, we recall that the mapping $A\colon {\mathbb R}^{n}\ni z\mapsto \lvert z\rvert^{p-2}z\in{\mathbb R}^{n}$ satisfies strict monotonicity,
\begin{equation}\label{strict monotonicity}
\textrm{i.e., }\langle A(z_{2})-A(z_{1})\mid z_{2}-z_{1} \rangle>0\quad\textrm{for all}\quad z_{1},\,z_{2}\in{\mathbb R}^{n}\quad\textrm{with}\quad z_{1}\neq z_{2}.
\end{equation}
\begin{proof}
We take arbitrary \(\delta>0\).
By \(u^{+},\,u^{-}\in C(\overline{\Omega})\) and (\ref{pointwise ineqality}), we can take a subdomain \(\Omega^{\prime}\Subset\Omega\) such that \(u^{-}\le u^{+}+\delta\) in \(\Omega\setminus\Omega^{\prime}\). This implies that the support of the truncated non-negative function \(w_{\delta}\coloneqq\mleft(u^{+}-u^{-}+\delta\mright)_{-}\in W^{1,\,p}(\Omega)\) is contained in \(\overline{\Omega^{\prime}}\subsetneq \Omega\) and therefore \(w_{\delta}\in D_{+}(\Omega)\). Let \(Z^{+},\,Z^{-}\in L^{\infty}(\Omega,\,{\mathbb R}^{n})\) be vector fields such that \((u^{+},\,Z^{+}),\,(u^{-},\,Z^{-})\) satisfies (\ref{weak subsolution})--(\ref{gradient vector field}), (\ref{gradient vector field})--(\ref{weak supersolution}) respectively. As in Remark \ref{some remarks on solutions}, we may test \(w_{\delta}\) in (\ref{weak subsolution}) and (\ref{weak supersolution}). Note that \(\nabla w_{\delta}=-\chi_{\delta}\nabla\mleft(u^{+}-u^{-}\mright)\), where \(\chi_{\delta}\) denotes the characteristic function of \(A_{\delta}\coloneqq\{x\in\Omega\mid u^{+}+\delta\le u^{-}\}\). Hence, we have 
\begin{align*}
0&\le -b\int_{A_{\delta}}\mleft\langle Z_{+}-Z_{-}\mathrel{}\middle|\mathrel{}\nabla u^{+}-\nabla u^{-}\mright\rangle \,dx-\int_{A_{\delta}}\mleft\langle \lvert\nabla u^{+}\rvert^{p-2}\nabla u^{+}-\lvert\nabla u^{-}\rvert^{p-2}\nabla u^{-} \mathrel{}\middle|\mathrel{}\nabla u^{+}-\nabla u^{-}\mright\rangle\,dx\\&\le -\int_{A_{\delta}}\mleft\langle \lvert\nabla u^{+}\rvert^{p-2}\nabla u^{+}-\lvert\nabla u^{-}\rvert^{p-2}\nabla u^{-} \mathrel{}\middle|\mathrel{}\nabla u^{+}-\nabla u^{-}\mright\rangle\,dx.
\end{align*}
Here we have invoked (\ref{gradient vector field}) and monotonicity of the subdifferential operator \(\partial\lvert\,\cdot\,\rvert\).
From (\ref{strict monotonicity}) we can easily check that \(\nabla u^{+}=\nabla u^{-}\) in \(A_{\delta}\), and therefore \(w_{\delta}=0\) in \(W_{0}^{1,\,p}(\Omega)\).
This means that \(u^{-}\le u^{+}+\delta\) a.e. in \(\Omega\). By regularity assumptions \(u^{+},\,u^{-}\in C(\overline{\Omega})\), we conclude that \(u^{-}\le u^{+}+\delta\) in \(\overline{\Omega}\). Since \(\delta>0\) is arbitrary, this completes the proof.
\end{proof}
\begin{remark}\upshape\label{Kruegel mistakes}\upshape
In 2013, Kr\"{u}gel gave another type of definitions of weak subsolutions and weak supersolutions to \(L_{b,\,p}=a\), where \(a\in{\mathbb R}\) is a constant. In Kr\"{u}gel's definition \cite[Definition 4.6]{krugel2013variational}, a function \(u^{-}\in W^{1,\,p}(\Omega)\) is called a \textit{subsolution} to \(L_{b,\,p}=a\) if \(u^{-}\) satisfies
\begin{equation}\label{kruegel def of subsolution}
\int_{D^{-}}\mleft\langle\frac{\nabla u^{-}}{\lvert \nabla u^{-}\rvert}\mathrel{}\middle|\mathrel{}\nabla\phi\mright\rangle\,dx+\int_{F^{-}}\lvert\nabla\phi\rvert\,dx +\int_{\Omega}\mleft\langle\lvert \nabla u^{-}\rvert^{p-2}\nabla u^{-}\mathrel{}\middle|\mathrel{}\nabla\phi\mright\rangle\,dx\le \int_{\Omega}a\phi\,dx
\end{equation}
for all \(\phi\in D_{+}(\Omega)\).
Here \(F^{-}\coloneqq \{x\in\Omega\mid\nabla u^{-}(x)=0\},\,D^{-}\coloneqq \Omega\setminus F^{-}\).
Similarly a function \(u^{+}\in W^{1,\,p}(\Omega)\) is called a \textit{supersolution} to \(L_{b,\,p}=a\) if \(u^{+}\) satisfies
\begin{equation}\label{kruegel def of supersolution}
\int_{D^{+}}\mleft\langle\frac{\nabla u^{+}}{\lvert \nabla u^{+}\rvert}\mathrel{}\middle|\mathrel{}\nabla\phi\mright\rangle\,dx+\int_{F^{+}}\lvert\nabla\phi\rvert\,dx +\int_{\Omega}\mleft\langle\lvert \nabla u^{+}\rvert^{p-2}\nabla u^{+}\mathrel{}\middle|\mathrel{}\nabla\phi\mright\rangle\,dx\ge \int_{\Omega}a\phi\,dx
\end{equation}
for all \(\phi\in D_{+}(\Omega)\).
Here \(F^{+}\coloneqq \{x\in\Omega\mid\nabla u^{+}(x)=0\},\,D^{+}\coloneqq \Omega\setminus F^{+}\).

The comparison principle discussed by Kr\"{u}gel \cite[Theorem 4.8]{krugel2013variational} states that
\begin{equation}\label{krugel theorem 4.8}
(u^{-}-u^{+})_{+}\in D_{+}(\Omega)\textrm{ implies }u^{-}\le u^{+}\quad\textrm{a.e. in }\Omega.
\end{equation}
By testing \((u^{-}-u^{+})_{+}\in D_{+}(\Omega)\) into (\ref{kruegel def of subsolution})(\ref{kruegel def of supersolution}) and substracting the two inequalities, Kr\"{u}gel claims that \(\nabla u^{-}=\nabla u^{+}\) over \(\Omega^{\prime}\coloneqq \{x\in\Omega\mid u^{-}(x)\ge u^{+}(x)\}\) and hence \(u^{-}=u^{+}\) a.e. in \(\Omega^{\prime}\). Despite Kr\"{u}gel's comment that integrals over \(F^{-}\) and \(F^{+}\) cancel out, however, it seems unclear whether
\begin{equation}\label{kruegel claim which seems to fail}
\int_{F^{-}}\lvert \nabla (u^{-}-u^{+})_{+}\rvert\,dx=\int_{F^{+}}\lvert \nabla (u^{-}-u^{+})_{+}\rvert\,dx
\end{equation}
is valid.
This problem is essentially due to the fact that Kr\"{u}gel did not appeal to monotonicity of the subdifferential operator \(\partial\lvert\,\cdot\,\rvert\) and did not regard the term \(\nabla u/\lvert\nabla u\rvert\) as an \(L^{\infty}\)-vector field satisfying the property (\ref{gradient vector field}). In our proof of the comparison principle (Proposition \ref{comparison principle}), we make use of monotonicity of the operator \(\partial\lvert\,\cdot\,\rvert\). Compared to our argument based on monotonicity, the equality (\ref{kruegel claim which seems to fail}) itself seems to be too strong to hold true.
\end{remark}
\subsection{Construction of classical subsolutions}\label{Subsect Construction of subsolutions}
In Section \ref{Subsect Construction of subsolutions}, we construct a classical subsolution to \(L_{b,\,p}u=0\) in an open annulus.
\begin{lemma}\label{Hopf Lemma}
Let \(c\in{\mathbb R}^{n}\setminus\{0\},\,m>0\). Then for each fixed open ball \(B_{R}(x_{\ast})\subset{\mathbb R}^{n}\), there exists a function \(h\in C^{\infty}({\mathbb R}^{n}\setminus\{x_{\ast}\})\) such that
\begin{equation}\label{boundary condition}
h=0\quad \textrm{on }\partial B_{R}(x_{\ast}),\quad 0\le h\le m\quad\textrm{on }\overline{E_{R}(x_{\ast})},
\end{equation}
\begin{equation}\label{partial nu v<0}
\partial_{\nu}h<0\quad \textrm{on }\partial B_{R}(x_{\ast}),
\end{equation}
\begin{equation}\label{almost flat}
\lvert \nabla h\rvert\le \frac{\lvert c\rvert}{2}\quad\textrm{in } E_{R}(x_{\ast}),
\end{equation}
\begin{equation}\label{Subsolution}
v(x)\coloneqq h(x)+\langle c\mid x\rangle \textrm{ satisfies}\quad L_{b,\,p}v\le 0\quad \textrm{in \(E_{R}(x_{\ast})\), in the classical sense}.
\end{equation}
Here \(E_{R}(x_{\ast})\coloneqq B_{R}(x_{\ast})\setminus \overline{B_{R/2}(x_{\ast})}\) is an open annulus, and \(\nu\) in (\ref{partial nu v<0}) denotes the exterior unit vector normal to \(B_{R}(x_{\ast})\).
\end{lemma}
Before proving Lemma \ref{Hopf Lemma}, we fix some notations on matrices. For a given \(n\times n\) matrix \(A\), we write \(\trace (A)\) as the trace of \(A\). We denote \({\mathbf 1}_{n}\) by the \(n\times n\) unit matrix. For column vectors \(x=(x_{i})_{i},\,y=(y_{i})_{i}\in{\mathbb R}^{n}\), we define a tensor $x\otimes y$, which is regarded as a real-valued $n\times n$ matrix \[x\otimes y\coloneqq (x_{i}y_{j})_{i,\,j}=\begin{pmatrix}x_{1}y_{1}&\cdots & x_{1}y_{n}\\ \vdots & \ddots &\vdots\\ x_{n}y_{1}&\cdots &x_{n}y_{n}\end{pmatrix}.\]

Assume that \(h\) satisfies (\ref{almost flat}). Then the triangle inequality implies that
\begin{equation}\label{a trick}
0<\frac{1}{2}\lvert c\rvert\le \lvert\nabla v\rvert\le \frac{3}{2}\lvert c\rvert\quad\textrm{in}\quad E_{R}(x_{\ast}).
\end{equation}
The estimate (\ref{a trick}) allows us to calculate \(L_{b,\,p}v\) in the classical sense over \(E_{R}(x_{\ast})\). By direct calculations we have 
\[-L_{b,\,p}v=+\divx\mleft(\nabla_{z}E(\nabla v)\mright)=\sum\limits_{i,\,j=1}^{n}\partial_{z_{i}z_{j}}E(\nabla v)\partial_{x_{i}x_{j}}v=\trace \mleft(\nabla_{z}^{2}E(\nabla v)\nabla^{2}h\mright)\quad \textrm{in }E_{R}(x_{\ast}).\]
We note that \(\nabla^{2}v=\nabla^{2} h\) by definition.
Here we recall a well-known result on Pucci's extremal operators. For given constants \(0<\lambda\le\Lambda<\infty\) and a fixed \(n\times n\) symmetric matrix \(M\), we define
\[{\mathcal M}^{-}(M,\,\lambda,\,\Lambda)\coloneqq \lambda\sum_{\lambda_{i}>0}\lambda_{i}+\Lambda\sum_{\lambda_{i}<0}\lambda_{i},\]
where \(\lambda_{i}\in{\mathbb R}\) are the eigenvalues of \(M\).
The following formula is a well-known result \cite[Remark 5.36]{MR3887613} ;
\[{\mathcal M}^{-}(M,\,\lambda,\,\Lambda)=\inf\mleft\{\trace(AM)\mathrel{}\middle|\mathrel{} A\in{\mathcal A}_{\lambda,\,\Lambda}\mright\},\]
where \({\mathcal A}_{\lambda,\,\Lambda}\) denotes the set of all symmetric matrices whose eigenvalues all belong to the closed interval \([\lambda,\,\Lambda]\).
By (\ref{a trick}) $L_{b,\,p}v$ is an uniformly elliptic operator in $E_{R}(x_{\ast})$. This enables us to find constants $0<\lambda\le \Lambda<\infty$, depending on $0<b<\infty,\,1<p<\infty,\,\lvert c\rvert>0$, such that $\nabla_{z}^{2}E(\nabla v)\in \lbrack \lambda,\,\Lambda\rbrack$ in $E_{R}(x_{\ast})$. Combining these results, it suffices to show that
\begin{equation}\label{eigenvalue calculation}
{\mathcal M}^{-}\left(\nabla^{2}h(x),\,\lambda,\,\Lambda\right)=\lambda\sum_{\lambda_{i}>0}\lambda_{i}(x)+\Lambda\sum_{\lambda_{i}<0}\lambda_{i}(x)>0\quad\textrm{for all }x\in E_{R}(x_{\ast}),
\end{equation}
where \(\lambda_{i}(x)\in{\mathbb R}\) denotes the eigenvalues of \(\nabla^{2}h(x)\).

Now we construct classical subsolutions. Our first construction is a modification of that by E. Hopf \cite{MR50126}.
\begin{proof}
Without loss of generality we may assume \(x_{\ast}=0\).
We define 
\begin{equation}\label{Hopf's barrier}
h(x)\coloneqq e^{-\alpha\lvert x\rvert^{2}}-e^{-\alpha R^{2}}\quad\textrm{for }x\in {\mathbb R}^{n}.
\end{equation}
Here \(\alpha=\alpha(b,\,n,\,p,\,\lvert c\rvert,\,R)>0\) is a sufficiently large constant to be chosen later. It is clear that \(0\le h(x)\le e^{-\alpha R^{2}/4}-e^{-\alpha R^{2}}\) in \(\overline{E_{R}(0)}\). We first let \(\alpha>0\) be so large that
\begin{equation}\label{first condition of alpha}
m e^{\alpha R^{2}}\ge e^{3\alpha R^{2}/4}-1.
\end{equation}
From (\ref{first condition of alpha}), we can easily check (\ref{boundary condition}).
By direct calculation we get
\[\nabla h(x)=-2\alpha e^{-\alpha\lvert x\rvert^{2}}x,\textrm{ and }\nabla^{2}h(x)=-2\alpha e^{-\alpha\lvert x\rvert^{2}}{\mathbf 1}_{n}+4\alpha^{2}e^{-\alpha\lvert x\rvert^{2}}x\otimes x\quad\textrm{for each }x\in{\mathbb R}^{n}.\]
From this result, (\ref{partial nu v<0}) is clear.
Also, we have
\[\lvert \nabla h(x)\rvert\le 2\alpha Re^{-\alpha R^{2}/4}\quad \textrm{for all }x\in E_{R}(0).\]
Let \(\alpha>0\) be so large that
\begin{equation}\label{second condition of alpha}
\alpha e^{-\alpha R^{2}/4}\le \frac{\lvert c\rvert}{4R},
\end{equation}
then we can check that $h$ satisfies (\ref{almost flat}).
Now we prove (\ref{Subsolution}) to complete the proof.
For \(x\not=0\), the eigenvalues of \(\nabla^{2}h(x)\) are given by
\[\mleft\{\begin{array}{ccc}
\lambda_{\parallel}(x)& \coloneqq &  4\alpha^{2}\lvert x\rvert^{2}e^{-\alpha\lvert x\rvert^{2}}-2\alpha e^{-\alpha\lvert x\rvert^{2}},   \\ 
\lambda_{\perp}(x) &\coloneqq & -2\alpha e^{-\alpha\lvert x\rvert^{2}}, \end{array}\mright. \textrm{ and the geometric multiplicities are }\mleft\{\begin{array}{c}1,\\ n-1. \end{array} \mright. \]
Assume that $\alpha$ satisfies
\begin{equation}\label{third condition of alpha}
\alpha> \frac{2}{R^{2}},
\end{equation}
so that $\lambda_{\parallel}>0>\lambda_{\perp}$ in $E_{R}(0)$. Therefore we get
\begin{align*}
{\mathcal M}^{-}\left(\nabla^{2}h(x),\,\lambda,\,\Lambda\right)&=\lambda\lambda_{\parallel}(x)+(n-1)\Lambda\lambda_{\perp}(x)=2\alpha e^{-\alpha\lvert x\rvert^{2}}\left[\lambda(2\alpha\lvert x\rvert^{2}-1)-(n-1)\Lambda \right]\\&\ge 2\alpha e^{-\alpha\lvert x\rvert^{2}}\left[\lambda\left(\frac{R^{2}}{2}\alpha-1\right)-(n-1)\Lambda \right].
\end{align*}
We can take sufficiently large \(\alpha=\alpha(\lvert c\rvert,\,m,\,n,\,R,\,\lambda,\,\Lambda)>0\) so that \(\alpha\) satisfies (\ref{eigenvalue calculation}) and (\ref{first condition of alpha})--(\ref{third condition of alpha}). For such constant \(\alpha>0\), the function \(v\) defined as in (\ref{Hopf's barrier}) satisfies (\ref{boundary condition})--(\ref{Subsolution}).
\end{proof}
It is possible to construct an alternative function \(h\in C^{\infty}({\mathbb R}^{n}\setminus\{x_{0}\})\) which satisfies (\ref{boundary condition})--(\ref{Subsolution}). We give another proof of Lemma \ref{Hopf Lemma}, which is derived from \cite[Chapter 2.8]{MR2356201}.
\begin{proof}
Without loss of generality we may assume \(x_{\ast}=0\).
We define 
\begin{equation}\label{Another Hopf's barrier}
h(x)\coloneqq \beta\mleft[\lvert x\rvert^{-\alpha}-R^{-\alpha}\mright]\quad\textrm{for }x\in {\mathbb R}^{n}\setminus\{0\}.
\end{equation}
We will later determine positive constants \(\alpha,\,\beta>0\), depending on \(b,\,m,\,n,\,p,\,\lvert c\rvert,\,R\). It is clear that \(0\le h(x)\le \beta R^{-\alpha}(2^{\alpha}-1)\) in \(\overline{E_{R}(0)}\). We first let \(\alpha,\,\beta>0\) satisfy 
\begin{equation}\label{first condition of alpha and beta}
\beta\le \frac{mR^{\alpha}}{2^{\alpha}-1}.
\end{equation}
Then \(h\) satisfies (\ref{boundary condition}). By direct calculation we get
\[\nabla h(x)=-\frac{\alpha\beta x}{\lvert x\rvert^{\alpha+2}},\textrm{ and }\nabla^{2}h(x)=\frac{\alpha\beta}{\lvert x\rvert^{\alpha+2}}\mleft[(\alpha+2)\frac{x\otimes x}{\lvert x\rvert^{2}}-{\mathbf 1}_{n}\mright]\]
for each \(x\in {\mathbb R}^{n}\setminus\{0\}\). The estimate (\ref{partial nu v<0}) is clear by this result. Also, we have
\[\lvert \nabla h(x)\rvert\le \frac{\alpha\beta}{(R/2)^{\alpha+1}}\quad \textrm{for all }x\in E_{R}(0).\]
Let \(\alpha,\,\beta>0\) satisfy
\begin{equation}\label{second condition of alpha and beta}
\beta\le \frac{\lvert c\rvert(R/2)^{\alpha+1}}{2\alpha},
\end{equation}
then we can check that $h$ satisfies (\ref{almost flat}).
Now we prove (\ref{Subsolution}) to complete the proof.
For \(x\not=0\), the eigenvalues of \(\nabla^{2}h(x)\) are given by
\[\mleft\{\begin{array}{ccc}
\lambda_{\parallel}(x)& \coloneqq & (\alpha+1)\alpha\beta\lvert x\rvert^{-\alpha-2},   \\ 
\lambda_{\perp}(x) &\coloneqq & -\alpha\beta\lvert x\rvert^{-\alpha-2}, \end{array}\mright. \textrm{ and the geometric multiplicities are }\mleft\{\begin{array}{c}1,\\ n-1. \end{array} \mright.\]
It is clear that \(\lambda_{\parallel}>0>\lambda_{\perp}\) in \({\mathbb R}^{n}\setminus\{0\}\), and therefore
\[{\mathcal M}^{-}\left(\nabla^{2}h(x),\,\lambda,\,\Lambda\right)=\alpha\beta\lvert x\rvert^{-\alpha-2}\left[(\alpha+1)\lambda-(n-1)\Lambda\right]\]
for all $x\in E_{R}(0)$.
We take and fix sufficiently large \(\alpha=\alpha(n,\,\lambda,\,\Lambda)>0\) so that \(\alpha\) satisfies (\ref{eigenvalue calculation}).
For such \(\alpha>0\), we choose sufficiently small \(\beta=\beta(\lvert c\rvert,\,R,\,\alpha)>0\) so that \(\beta\) satisfies (\ref{first condition of alpha and beta})--(\ref{second condition of alpha and beta}). Then the function \(h\) defined as in (\ref{Another Hopf's barrier}) satisfies (\ref{boundary condition})--(\ref{Subsolution}).
\end{proof}

\subsection{Strong maximum principle}\label{Subsect Strong Maximum Principle}
We prove the strong maximum principle (Theorem \ref{strong maximum principle}).
\begin{proof}
Let \(D_{0}\subset D\) be a connected component of the open set \(D\), and \(x_{0}\in D_{0}\).
Without loss of generality we may assume that \(x_{0}=0\) and \(u(0)=0\). By Lemma \ref{regularity at regular point}, it is clear that \(\partial u(0)=\{\nabla u(0)\}\not=\{0\}\).
We set a vector \(c\coloneqq \nabla u(0)\in {\mathbb R}^{n}\setminus\{0\}\) and a relatively closed set
\[\Sigma\coloneqq \{x\in D_{0}\mid u(x)=\langle c\mid x\rangle\}.\]
and we will prove that \(\Sigma=D_{0}\).
It is also clear that \(0\in \Sigma\) and hence \(\Sigma\not= \emptyset\).
Suppose for contradiction that \(\Sigma\subsetneq D_{0}\). Then it follows that \(\partial\Sigma\cap D_{0}\not=\emptyset\), since \(D_{0}\) is connected. 
We may take and fix a point \(x_{\ast}\in D_{0}\setminus \Sigma\) such that
\(\dist(x_{\ast},\,\Sigma)<\dist(x_{\ast},\,\partial D_{0})\).
By extending a closed ball centered at \(x_{\ast}\) until it hits \(\Sigma\), we can take a point \(y_{\ast}\in D_{0}\) and a closed ball \(\overline{B_{R}(x_{\ast})}\subset D_{0}\) such that \(y_{\ast}\in\partial B_{R}(x_{\ast})\cap \Sigma\) and \(u(x)>\langle c\mid x\rangle\) for all \(x \in B_{R}(x_{\ast})\).
We note that
\begin{equation}\label{boundary value}
\left\{\begin{array}{cccc}
0& = & \min\limits_{x\in\partial B_{R}(x_{\ast})} \mleft(u(x)-\langle c\mid x\rangle\mright), & \textrm{achieved at \(y_{\ast}\in\partial B_{R}(x_{\ast})\)}, \\
m & \coloneqq & \min\limits_{x\in\partial B_{R/2}(x_{\ast})} \mleft(u(x)-\langle c\mid x\rangle\mright)>0, 
\end{array} \right.
\end{equation}
by construction of \(B_{R}(x_{\ast})\). Let \(h\in C^{\infty}({\mathbb R}^{n}\setminus \{x_{\ast}\})\) be an auxiliary function as in Lemma \ref{Hopf Lemma}. Then from (\ref{boundary value}) it is easy to check that \(v\coloneqq h+\langle c\mid x\rangle\) satisfies \(v\le u\) on \(\partial E_{R}(x_{\ast})\), in the sense of (\ref{pointwise ineqality}). By Proposition \ref{comparison principle}, we have \(v\le u\) on \(\overline{E_{R}(x_{\ast})}\). Hence \(0\le u-\langle c\mid x\rangle-h\) in \(\overline{E_{R}(x_{\ast})}\). This inequality becomes equality at \(y_{\ast}\in\partial E_{R}(x_{\ast})\) by (\ref{boundary condition}) and (\ref{boundary value}). Therefore the function \(u(x)-\langle c\mid x\rangle-h(x)\,(x\in \overline{E_{R}(x_{\ast})})\) takes its minimum at \(y_{\ast}\in \partial B_{R}(x_{\ast})\). Also by \(y_{\ast}\in \Sigma\) and the subgradient inequality \[u(x)\ge \langle c\mid x\rangle\quad\textrm{for all }x\in \Omega,\] it is clear that the function \(w(x)\coloneqq u(x)-\langle c\mid x\rangle\,(x\in D_{0})\) takes its minimum \(0\) at \(y_{\ast}\in D_{0}\). We note that \(w,\,w-h\in C^{1}(D_{0})\) by Lemma \ref{regularity at regular point}. By calculating classical partial derivatives at \(y_{\ast}\) in the direction \(\nu_{0}\coloneqq (y_{\ast}-x_{\ast})/R\), we obtain
\[0\ge \partial_{\nu_{0}}(w-h)(y_{\ast})=-\partial_{\nu_{0}}h(y_{\ast})>0.\]
This is a contradiction, and therefore \(\Sigma=D_{0}\).
\end{proof}
\section{Proofs of main theorems}\label{Sect Liouville Theorem}
In Section \ref{Sect Liouville Theorem}, we give proofs of the Liouville-type theorem (Theorem \ref{Liouville theorem}) and the \(C^{1}\)-regularity theorem (Thorem \ref{C1 regularity}).
\subsection{Liouville-type theorem}\label{Subsect Liouville}
For a preparation, we prove Lemma \ref{determination of the shape of the facet} below.
\begin{lemma}\label{determination of the shape of the facet}
Let $u$ be a real-valued convex function in ${\mathbb R}^{n}$. Assume that \(u\) satisfies the following,
\begin{enumerate}
\item The facet of $u$, $F\subset{\mathbb R}^{n}$, satisfies $\emptyset\subsetneq F\subsetneq {\mathbb R}^{n}$.
\item $u$ attains its minimum $0$.
\item $u$ is affine in each connected component of $D\coloneqq {\mathbb R}^{n}\setminus F$. \label{Condition affine on connected comp}
\end{enumerate}
Then up to a rotation and a shift translation, $u$ can be expressed as either of the following three types of piecewise-linear functions.
\begin{equation}\label{Possible Convex Solution Type 1}
u(x)=\max\{\,t_{1}x_{1},\,0\,\} \quad\textrm{for all }x\in{\mathbb R}^{n},
\end{equation}
\begin{equation}\label{Possible Convex Solution Type 2}
u(x)=\max\{\,t_{1}x_{1},\,-t_{2}x_{1}\,\}\quad\textrm{for all }x\in{\mathbb R}^{n},
\end{equation}
\begin{equation}\label{Possible Convex Solution Type 3}
u(x)=\max\{\,t_{1}x_{1},\,0,\,-t_{2}(x_{1}+l_{0})\,\}\quad \textrm{for all }x\in{\mathbb R}^{n}.
\end{equation}
Here $t_{1},\,t_{2},\,l>0$ are constants.
\end{lemma}
Before starting the proof of Lemma \ref{Liouville theorem}, we introduce notations on affine hyperplanes. For \(c\in{\mathbb R}^{n}\setminus\{0\}\) and \(x_{0}\in {\mathbb R}^{n}\), we define 
\[\left\{\begin{array}{ccc}
H_{c,\,x_{0}}&\coloneqq& \mleft\{x\in{\mathbb R}^{n}\mathrel{} \middle|\mathrel{}\langle c\mid x-x_{0}\rangle =0\mright\},\\
H_{c,\,x_{0}}^{-}&\coloneqq& \mleft\{x\in{\mathbb R}^{n}\mathrel{} \middle|\mathrel{}\langle c\mid x-x_{0}\rangle < 0\mright\},\\
H_{c,\,x_{0}}^{+}&\coloneqq& \mleft\{x\in{\mathbb R}^{n}\mathrel{} \middle|\mathrel{}\langle c\mid x-x_{0}\rangle > 0\mright\}.\\
\end{array}\right.\]
In order to prove the Liouville-type theorem, we will make use of the supporting hyperplane theorem, which states that for any non-empty closed convex set \(C\subset{\mathbb R}^{n}\) and \(x_{0}\in\partial C\), there exists \(c\in{\mathbb R}^{n}\setminus\{0\}\) such that
\[\sup\limits_{x\in C}\,\langle c\mid x\rangle\le \langle c\mid x_{0}\rangle,\quad\textrm{and in particular}\quad H_{c,\,x_{0}}^{+}\subset {\mathbb R}^{n}\setminus C.\]
For such \(c\in{\mathbb R}^{n}\setminus\{0\}\), a hyperplane \(H_{c,\,x_{0}}\) is often called a supporting hyperplane for \(C\) at the boundary point \(x_{0}\).
For the proof of the supporting hyperplane theorem, see \cite[Proposition 1.5.1]{MR2830150}.

\begin{proof}
Since ${\mathbb R}^{n}$ is connected and $F\subset{\mathbb R}^{n}$ is a closed convex set, it follows that \(\partial F\not= \emptyset\).
Without loss of generality we may assume that \(0\in\partial F\) and \(u(0)=0\).

By the supporting hyperplane theorem, we can take and fix a supporting hyperplane for \(F\) at the boundary point \(0\), which we write \(H_{c,\,0}\subset {\mathbb R}^{n}\). By rotation, we may assume that \(c=e_{1}\). Let \(D_{1}\) be the connected component of \(D\) which contains \(H_{e_{1},\,0}^{+}\subset{\mathbb R}^{n}\setminus F=D\).
By the assumption \ref{Condition affine on connected comp} and \(u(0)=0\), it follows that there exists \(c\in {\mathbb R}^{n}\setminus\{0\}\) such that \(u(x)=\langle c\mid x\rangle\) for all \(x\in D_{1}\). We should note that \(H_{c,\,0}=H_{e_{1},\,0}\) and hence \(c=t_{1}e_{1}\) for some \(t_{1}\in(0,\,\infty)\), since otherwise it follows that \(H_{e_{1},\,0}^{+}\cap H_{c,\,0}^{-}\not=\emptyset\) and \(0\le u(x_{0})=\langle c\mid x_{0}\rangle<0\) for any \(x_{0}\in H_{e_{1},\,0}^{+}\cap H_{c,\,0}^{-}\). The result \(H_{c,\,0}=H_{e_{1},\,0}\) also implies that \(H_{e_{1},\,0}\subset \partial F\subset F\subset \{x\in{\mathbb R}^{n}\mid x_{1}\le 0\}=H_{e_{1},\,0}^{-}\cup H_{e_{1},\,0}\). Now we will deduce three possible representations of \(u\).

If \(\partial F=H_{e_{1},\,0}\), then we have either \(F=H_{e_{1},\,0}^{-}\cup H_{e_{1},\,0}\) or \(F=H_{e_{1},\,0}\), since the open set \(H_{e_{1},\,0}^{-}=\{x\in{\mathbb R}^{n}\mid x_{1}<0\}\) is connected. For the first case, \(u\) is clearly expressed by (\ref{Possible Convex Solution Type 1}).
For the second case, it is clear that \(D\) consists of two connected components \(D_{1}=H_{e_{1},\,0}^{+}\) and \(D_{2}=H_{e_{1},\,0}^{-}\). Again by the condition \ref{Condition affine on connected comp} and similar arguments to the above, we can determine \(u|_{D_{2}}\) as \(u(x)=\langle -t_{2}e_{1}\mid x\rangle\) for all \(x\in D_{2}\). Here \(t_{2}\in (0,\,\infty)\) is a constant. Hence we obtain (\ref{Possible Convex Solution Type 2}).
For the case \(H_{e_{1},\,0}\subsetneq \partial F\), we take and fix \(z_{0}\in \partial F\setminus H_{e_{1},\,0}\) and a supporting hyperplane for \(F\) at \(z_{0}\), which we write by \(H_{c^{\prime},\,z_{0}}\). Let \(D_{2}\) be the connected component of \(D\) which contains \(H_{c^{\prime},\,z_{0}}^{+}\subset D\). By the assumption \ref{Condition affine on connected comp} and \(u(z_{0})=0\), it follows that there exists \(c^{\prime\prime}\in {\mathbb R}^{n}\setminus\{0\}\) such that \(u(x)=\mleft\langle c^{\prime\prime}\mathrel{}\middle|\mathrel{} x-z_{0}\mright\rangle\) for all \(x\in D_{2}\). Completely similarly to the arguments above for showing that \(H_{c,\,0}=H_{e_{1},\,0}\), we can easily notice that \(H_{c^{\prime\prime},\,z_{0}}=H_{c^{\prime},\,z_{0}}\) and hence \(c^{\prime\prime}=t_{1}^{\prime}c^{\prime}\) for some constant \(t_{1}^{\prime}\in(0,\,\infty)\).
Moreover, we also realize that \(c^{\prime}=t_{\ast}e_{1}\) for some \(t_{\ast}\in{\mathbb R}\setminus\{0\}\). Otherwise it follows that the two hyperplanes \(H_{e_{1},\,0}\) and \(H_{c^{\prime},\,z_{0}}\) cross, and hence we get \(D_{1}=D_{2}\) and \(H_{e_{1},\,0}^{+}\cap H_{c^{\prime},\,z_{0}}^{-}\not= \emptyset\), which implies that there exists a point \(x_{0}\in D\) such that \(u(x_{0})<0\). This is clearly a contradiction.
This result and convexity of \(u\) imply that \(D\) consists of two connected components \(D_{1}=H_{e_{1},\,0}^{+}\) and \(D_{2}=H_{-e_{1},\,z_{0}}^{+}\), and that \(F=\{x\in{\mathbb R}^{n}\mid -l_{0}\le x_{1}\le 0\}\). Here \(l_{0}\coloneqq \dist(H_{e_{1},\,0},\,H_{-e_{1},\,z_{0}})>0\).
Finally we obtain the last possible expression (\ref{Possible Convex Solution Type 3}).
\(u\) can be expressed by either of (\ref{Possible Convex Solution Type 1})--(\ref{Possible Convex Solution Type 3}).
\end{proof}
Now we give the proof of Theorem \ref{Liouville theorem}.
\begin{proof}
Assume by contradiction that $F$, the facet of $u$, would satisfy $\emptyset\subsetneq F\subsetneq {\mathbb R}^{n}$. Without loss of generality, we may assume that $u$ attains its minimum $0$. By the strong maximum principle (Theorem \ref{strong maximum principle}), the convex weak solution $u$ is affine in each connected component of $D\coloneqq {\mathbb R}^{n}\setminus F$. Therefore we are able to apply Lemma \ref{determination of the shape of the facet}. By rotation and translation, $u$ can be expressed as (\ref{Possible Convex Solution Type 1})--(\ref{Possible Convex Solution Type 3}).
Now we prove that \(u\) is no longer a weak solution to \(L_{b,\,p}u=0\) in \({\mathbb R}^{n}\).
We set open cubes \(Q^{\prime}\coloneqq (-1,\,1)^{n-1}\subset{\mathbb R}^{n-1}\) and \(Q\coloneqq (-d,\,d) \times Q^{\prime}\subset {\mathbb R}^{n}\), where \(d>0\) is to be chosen later. We claim that \(u\) does not satisfy \(L_{b,\,p}u=0\) in \(W^{-1,\,p^{\prime}}(Q)\). Assume by contradiction that there exists a vector field \(Z\in L^{\infty}(Q,\,{\mathbb R}^{n})\) such that the pair \((u,\,Z)\in W^{1,\,p}(Q)\times L^{\infty}(Q,\,{\mathbb R}^{n})\) satisfies $L_{b,\,p}u=0$ in $W^{-1,\,p^{\prime}}(Q)$.

For the first case (\ref{Possible Convex Solution Type 1}), we have
\begin{equation}\label{subgradient condition for type 1}
\lvert Z(x)\rvert\le 1\quad\textrm{for a.e. }x\in Q,\,\quad \textrm{and } Z(x)=e_{1}\quad\textrm{for a.e. }x\in Q_{r}\coloneqq (0,\,d)\times Q^{\prime}\subset{\mathbb R}^{n}.
\end{equation}
by definition of \(Z\). We also set another open cube \(Q_{l}\coloneqq (-d,\,0)\times Q^{\prime}\subset{\mathbb R}^{n}\).
We take and fix non-negative functions \(\phi_{1}\in C_{c}^{1}((-d,\,d)),\,\phi_{2}\in C_{c}^{1}\mleft(Q^{\prime}\mright)\) such that
\begin{equation}\label{test function in 2 variables}
\phi_{1}^{\prime}\ge 0 \textrm{ in }(-d,\,0),\,\max\limits_{(-d,\,d)}\phi_{1}=\phi_{1}(0)>0, \textrm{ and }\phi_{2}\not\equiv 0.
\end{equation}
We define an admissible test function \(\phi\in C_{c}^{1}(Q)\) by \(\phi(x_{1},\,x^{\prime})\coloneqq \phi_{1}(x_{1})\phi_{2}(x^{\prime})\) for \((x_{1},\,x^{\prime})\in (-d,\,d) \times Q^{\prime}=Q\). Test \(\phi\in C_{c}^{1}(Q)\) into \(L_{b,\,p}u=0\) in $W^{-1,\,p^{\prime}}(Q)$, and divide the integration over $Q$ into that over $Q_{l}$ and $Q_{r}$. Then (\ref{subgradient condition for type 1}) implies that
\begin{align*}
0&=b\int_{Q_{l}}\langle Z+\lvert 0\rvert^{p-2}0\mid \nabla(\phi_{1}\phi_{2})\rangle\,dx+\int_{Q_{r}}\mleft\langle (b+t_{1}^{p-1})e_{1}\mathrel{}\middle| \mathrel{} \nabla(\phi_{1}\phi_{2})\mright\rangle\,dx\\&\le b\int_{Q_{l}}\phi_{1}^{\prime}\phi_{2}\,dx+b\int_{Q_{l}}  \phi_{1}\lvert\nabla\phi_{2}\rvert \,dx\\&\quad +b\phi_{1}(0)\int_{Q^{\prime}}\phi_{2}(x^{\prime})\langle e_{1}\mid -e_{1}\rangle\,dx^{\prime}+t_{1}^{p-1}\phi_{1}(0)\int_{Q^{\prime}}\phi_{2}(x^{\prime})\langle e_{1}\mid-e_{1}\rangle\,dx^{\prime}\\&\eqqcolon I_{1}+I_{2}+I_{3}+I_{4}.
\end{align*}
Here we have applied the Gauss--Green theorem to the integration over $Q_{r}$, and the Cauchy--Schwarz inequality to the integration over $Q_{l}$. 
For the integrations $I_{1}$ and $I_{2}$, we make use of Fubini's theorem and (\ref{test function in 2 variables}). Then we have
\[I_{1}=\int_{Q^{\prime}}\left(\int_{-d}^{0}\phi_{1}^{\prime}(x_{1})dx_{1}\right)\phi_{2}(x^{\prime})dx^{\prime}=b\phi_{1}(0)\int_{Q^{\prime}}\phi_{2}(x^{\prime})dx^{\prime}=b\phi_{1}(0)\lVert \phi_{2}\rVert_{L^{1}(Q^{\prime})}=-I_{3},\]
\[I_{2}\le b\phi_{1}(0)\int_{-d}^{0}\,dx_{1}\int_{Q^{\prime}}\lvert\nabla \phi_{2}(x^{\prime})\rvert\,dx^{\prime}=bd\phi_{1}(0)\lVert \nabla\phi_{2}\rVert_{L^{1}(Q^{\prime})}.\]
Finally we obtain
\begin{equation}\label{an estimate for Type 1}
0\le I_{1}+I_{2}+I_{3}+I_{4}\le I_{2}+I_{4}\le \phi_{1}(0)\left(bd\lVert\nabla\phi_{2}\rVert_{L^{1}(Q^{\prime})}-t_{1}^{p-1}\lVert \phi_{2}\rVert_{L^{1}(Q^{\prime})}\right).
\end{equation}
From (\ref{an estimate for Type 1}), we can easily deduce a contradiction by choosing sufficiently small \(d=d(b,\,p,\,t_{1},\,\phi_{2})>0\).
Similarly we can prove that $u$ defined as in (\ref{Possible Convex Solution Type 3}) does not satisfy $L_{b,\,p}u=0$ in $W^{-1,\,p^{\prime}}(Q)$, since it suffices to restrict \(d<l_{0}\). We consider the remaining case (\ref{Possible Convex Solution Type 2}). We have
\[Z(x)=\left\{\begin{array}{cc}
e_{1} & \textrm{for a.e. }x\in Q_{r},\\
-e_{1} & \textrm{for a.e. }x\in Q_{l}.
\end{array} \right.\]
by definition of \(Z\).
We test the same function \(\phi\in C_{c}^{1}(Q)\) in \(L_{b,\,p}u=0\), then it follows that
\begin{align*}
0&=\int_{Q_{l}}\mleft\langle -(b+t_{2}^{p-1})e_{1}\mathrel{}\middle| \mathrel{} \nabla(\phi_{1}\phi_{2})\mright\rangle\,dx+\int_{Q_{r}}\mleft\langle (b+t_{1}^{p-1})e_{1}\mathrel{}\middle| \mathrel{} \nabla(\phi_{1}\phi_{2})\mright\rangle\,dx\\&=-(b+t_{2}^{p-1})\int_{Q^{\prime}}\phi_{1}(0)\phi_{2}(x^{\prime})\mleft\langle e_{1}\mathrel{}\middle| \mathrel{} e_{1}\mright\rangle\,dx^{\prime}+(b+t_{1}^{p-1})\int_{Q^{\prime}}\phi_{1}(0)\phi_{2}(x^{\prime})\mleft\langle e_{1}\mathrel{}\middle| \mathrel{} -e_{1}\mright\rangle\,dx^{\prime}\\&=-\phi_{1}(0)\mleft(2b+t_{1}^{p-1}+\lvert t_{2}\rvert^{p-1}\mright)\int_{Q^{\prime}}\phi_{2}(x^{\prime})\,dx^{\prime}<0,
\end{align*}
which is a contradiction. This completes the proof.
\end{proof}
\begin{remark}\upshape
The estimate (\ref{an estimate for Type 1}) breaks for $p=1$, since the equation $\lvert 0\rvert^{p-2}0=0$ is no longer valid for $p=1$. This means that we have implicitly used differentiability of the function $\lvert z\rvert^{p}/p$ at $0\in{\mathbb R}^{n}$. Also it should be noted that for the one-variable case, functions as in (\ref{Possible Convex Solution Type 1}), which are in general not in \(C^{1}\), are one-harmonic in ${\mathbb R}$.
\end{remark}
\subsection{\(C^{1}\)-regularity theorem}
We give the proof of Theorem \ref{C1 regularity}.
\begin{proof}
We may assume that \(\Omega\) is convex.
By \cite[Theorem 25.1 and 25.5]{MR1451876} and Lemma \ref{regularity at regular point}, it suffices to show that \(\partial u(x_{0})=\{0\}\) for all \(x_{0}\in F\). Let \(x_{0}\in F\).
We get a convex function \(u_{0}\colon {\mathbb R}^{n}\rightarrow {\mathbb R}\) as a blow-up limit as in Proposition \ref{limit of rescaled solutions}.
We note that the facet of \(u_{0}\) is non-empty by Proposition \ref{limit of rescaled solutions}.
Hence by the Liouville-type theorem (Theorem \ref{Liouville theorem}), $u_{0}$ is constant and we obtain \(\partial u_{0}(x_{0})=\{0\}\).
Combining these results, we have \(\{0\}\subset \partial u(x_{0})\subset \partial u_{0}(x_{0})=\{0\}\) and therefore \(\partial u(x_{0})=\{0\}\). This completes the proof.
\end{proof}
\section{Generalization}\label{Sect Generalization}
In Section \ref{Sect Generalization}, we would like to discuss $C^{1}$-regularity of convex weak solutions to
\begin{equation}\label{generalized eq}
Lu\coloneqq -\divx(\nabla_{z}\Psi(\nabla u))-\divx(\nabla_{z}W(\nabla u))=f\quad\textrm{in}\quad \Omega\subset{\mathbb R}^{n},
\end{equation}
which covers (\ref{crystal model eq again}).
Precisely speaking, throughout Section \ref{Sect Generalization}, we make these following assumptions for $\Psi$ and $W$ on regularity and ellipticity.
For regularity, we only require
\begin{equation}\label{regularity assumptions}
\Psi\in C({\mathbb R}^{n})\cap C^{2}({\mathbb R}^{n}\setminus\{ 0\}),\,W\in C^{1}({\mathbb R}^{n})\cap C^{2}({\mathbb R}^{n}\setminus\{ 0\}).
\end{equation}
For $W$, we assume that for each fixed \(0<\mu\le M<\infty\), there exist constants \(0<\gamma<\Gamma<\infty\) such that \(W\) satisfies
\begin{equation}\label{ellipticity on W}
\gamma \lvert \zeta\rvert^{2} \le \mleft\langle \nabla_{z}^{2}W(z_{0})\zeta\mathrel{}\middle|\mathrel{}\zeta \mright\rangle,
\end{equation}
\begin{equation}\label{bound on W}
\mleft\lvert \mleft\langle \nabla_{z}^{2}W(z_{0})\zeta\mathrel{} \middle|\mathrel{} \omega \mright\rangle \mright\rvert \le \Gamma \lvert \zeta\rvert\lvert \omega\rvert
\end{equation}
for all \(z_{0},\,\zeta,\,\omega\in{\mathbb R}^{n}\) with \(\mu\le \lvert z_{0}\rvert\le M\). Also, there is no loss of generality in assuming that
\begin{equation}\label{gradW at 0}
\nabla_{z}W(0)=0.
\end{equation}
Finally, we assume that \(\Psi\) is positively homogeneous of degree $1$. In other words, \(\Psi\) satisfies
\begin{equation}\label{hom of deg 1}
\Psi(\lambda z_{0})=\lambda\Psi(z_{0})
\end{equation}
holds for all \(z_{0}\in{\mathbb R}^{n}\) and \(\lambda>0\). This clearly yields \(\Psi(0)=0\).

By modifying some of our arguments, we are able to show that
\begin{theorem}[$C^{1}$-regularity theorem for general equations]\label{C1 theorem generalized}
Let \(\Omega\subset{\mathbb R}^{n}\) be a domain. Assume that \(f\in L_{\mathrm{loc}}^{q}(\Omega)\,(n<q\le\infty)\) and the functionals \(\Psi\) and \(W\) satisfy (\ref{regularity assumptions})--(\ref{gradW at 0}). If \(u\) is a convex weak solution to (\ref{generalized eq}), then \(u\) is in \(C^{1}(\Omega)\).
\end{theorem}
If we set
\[\Psi(z)\coloneqq b\lvert z\rvert,\quad W(z)\coloneqq \frac{\lvert z\rvert^{p}}{p},\quad\textrm{where }1<p<\infty,\]
then the equation (\ref{generalized eq}) becomes (\ref{crystal model eq again}). Therefore Theorem \ref{C1 theorem generalized} generalizes Theorem \ref{C1 regularity}.

\subsection{Preliminaries}\label{Subsect Some properties}
In Section \ref{Subsect Some properties}, we mention some basic properties of $\Psi$ and $W$, which are derived from the assumptions (\ref{regularity assumptions})--(\ref{gradW at 0}).

For $W$, by (\ref{regularity assumptions})--(\ref{ellipticity on W}) and (\ref{gradW at 0}) it is easy to check that the continuous mapping \(A\colon {\mathbb R}^{n}\ni z\mapsto \nabla W(z)\in {\mathbb R}^{n}\) satisfies strict monotonicity (\ref{strict monotonicity}).
In particular, by (\ref{gradW at 0}) we have
\begin{equation}\label{claim for gradW}
\langle A(z)\mid z \rangle>0\quad\textrm{for all }z\in{\mathbb R}^{n}\setminus\{ 0\}.
\end{equation}
For the proof, see Lemma \ref{lemma strict monotonicity} in the appendices.

For \(\Psi\), we first note that \(\Psi\) satisfies the triangle inequality
\begin{equation}\label{triangle ineq}
\Psi(z_{1}+z_{2})\le \Psi(z_{1})+\Psi(z_{2})\quad\textrm{for all }z_{1},\,z_{2}\in{\mathbb R}^{n}.
\end{equation}
We define a function \({\tilde \Psi}\colon {\mathbb R}^{n}\rightarrow [0,\,\infty]\) by
\[{\tilde \Psi}(\zeta)\coloneqq \sup\mleft\{ \langle \zeta\mid z\rangle \mathrel{} \middle| \mathrel{} z\in{\mathbb R}^{n},\,\Psi(z)\le 1 \mright\}.\]
\({\tilde \Psi}\) is the support function for the closed convex set \(C_{\Psi}\coloneqq\mleft\{ z \in{\mathbb R}^{n}\mathrel{} \middle|\mathrel{} \Psi(z)\le 1\mright\}\). By definition it is easy to check that \({\tilde \Psi}\) is convex and lower semicontinuous. Also, if \(\zeta\in{\mathbb R}^{n}\) satisfies \({\tilde \Psi}(\zeta)<\infty\), then the following Cauchy--Schwarz-type inequality holds;
\begin{equation}\label{CW ineq}
\langle z\mid \zeta \rangle\le \Psi(z){\tilde \Psi}(\zeta)\quad \textrm{for all}\quad z\in{\mathbb R}^{n}.
\end{equation}
If a convex function \(\Psi\) is positively homogeneous of degree $1$, then the subdifferential operator \(\partial \Psi\) is explicitly given by
\begin{equation}\label{subdifferential expression}
\partial \Psi(z)=\mleft\{ \zeta \in{\mathbb R}^{n}\mathrel{} \middle| \mathrel{} {\tilde \Psi}(\zeta)\le 1,\,\Psi(z)=\langle z\mid \zeta\rangle \mright\}
\end{equation}
for all \(z\in{\mathbb R}^{n}\). In particular, we have the following formula
\begin{equation}\label{Euler's id}
\mleft\langle \nabla_{z}\Psi(z_{0})\mathrel{}\middle|\mathrel{} z_{0}\mright\rangle=\Psi(z_{0})\quad\textrm{for all }z_{0}\in {\mathbb R}^{n}\setminus\{ 0\},
\end{equation}
which is often called Euler's identity.
Also, assumptions (\ref{regularity assumptions}) and (\ref{hom of deg 1}) imply that
\begin{equation}\label{gradient and Hessian for homogeneous}
\nabla \Psi (\lambda z_{0})=\nabla \Psi(z_{0}),\quad \nabla^{2} \Psi (\lambda z_{0})=\lambda^{-1}\nabla^{2}\Psi(z_{0})
\end{equation}
for all \(\lambda >0\) and \(z_{0}\in {\mathbb R}^{n}\setminus\{ 0\}\). 
Proofs of (\ref{triangle ineq})--(\ref{subdifferential expression}) are given in Lemma \ref{basic lemma of convex deg 1} of the appendices for the reader's convenience.
\begin{remark}\upshape\label{boundedness of subdifferential}
The results (\ref{Euler's id})--(\ref{gradient and Hessian for homogeneous}) give us the following basic property for $\Psi$.
\begin{enumerate}
\item We set a constant \[K\coloneqq\sup\mleft\{\lvert \nabla_{z}\Psi(z_{0})\rvert\mathrel{}\middle|\mathrel{}z_{0}\in{\mathbb R}^{n},\,\lvert z_{0}\rvert=1\mright\},\]
which is finite. Then we have \(\partial\Psi(z_{0})\subset \overline{B_{K}(0)}\) for all \(z_{0}\in{\mathbb R}^{n}\). For the case \(z_{0}\not= 0\), this inclusion is clear by (\ref{gradient and Hessian for homogeneous}) and \(\partial\Psi(z_{0})=\{\nabla_{z}\Psi(z_{0})\}\). For \(z_{0}=0\), we take arbitrary \(w\in \partial \Psi(0)\setminus\{ 0\}\). Then by the subgradient inequality, Euler's identity (\ref{Euler's id}) and the Cauchy--Schwarz inequality, we have
\begin{align*}
\lvert w\rvert^{2}&=\langle w\mid w-0 \rangle+\Psi(0)\\&\le \Psi(w)=\mleft\langle \nabla_{z}\Psi(w)\mathrel{}\middle|\mathrel{} w\mright\rangle\le K\lvert w\rvert.
\end{align*}
This estimate yields the inclusion \(\partial\Psi(0)\subset \overline{B_{K}(0)}\). 
\item For $z_{0}\in{\mathbb R}^{n}\setminus\{ 0\}$, the Hessian matrix $\nabla_{z}^{2}\Psi(z_{0})$ satisfies 
\begin{equation}\label{degenerate ellipticity for 1 degree}
0\le \mleft\langle \nabla_{z}^{2}\Psi(z_{0})\zeta\mathrel{}\middle| \mathrel{}\zeta \mright\rangle,
\end{equation}
\begin{equation}\label{bounds for 1 degree}
\mleft\lvert\mleft\langle \nabla_{z}^{2}\Psi(z_{0})\zeta\mathrel{}\middle| \mathrel{}\omega \mright\rangle\mright\rvert\le \frac{C}{\lvert z_{0}\rvert}\lvert \zeta\rvert\lvert \omega\rvert
\end{equation}
for all \(\zeta,\,\omega\in{\mathbb R}^{n}\). Here the finite constant \(C\) is explicitly given by
\[C\coloneqq \sup\mleft\{\mleft\lvert\mleft\langle \nabla_{z}^{2}\Psi(w)\zeta\mathrel{}\middle| \mathrel{}\omega \mright\rangle\mright\rvert \mathrel{}\middle|\mathrel{}z,\,\zeta,\,\omega\in {\mathbb R}^{n},\,\lvert w\rvert=\lvert \eta\rvert=\lvert \omega\rvert=1\mright\}.\]
\end{enumerate}
\end{remark}
Lemma \ref{Convergence lemma in L-infty} states lower semicontinuity of a functional in the weak$^{\ast}$ topology of an \(L^{\infty}\)-space. This result is used in the justification of a blow-up argument for the equation (\ref{generalized eq}).
\begin{lemma}\label{Convergence lemma in L-infty}
Let \(\Omega\subset {\mathbb R}^{m}\) be a Lebesgue measurable set, and let \(\Psi\colon {\mathbb R}^{n}\rightarrow [0,\,\infty)\) be a convex function which satisfies (\ref{hom of deg 1}).
Assume that a vector field \(Z\in L^{\infty}(\Omega,\,{\mathbb R}^{n})\) and a sequence \(\{Z_{N}\}_{N}\subset L^{\infty}(\Omega,\,{\mathbb R}^{n})\) satisfy \(Z_{N}\overset{\ast}{\rightharpoonup} Z\) in \(L^{\infty}(\Omega,\,{\mathbb R}^{n})\).
Then we have 
\begin{equation}\label{weakstar lower semicontinuity lemma}
\esssup\limits_{x\in\Omega} {\tilde \Psi}(Z(x))\le \liminf_{N\to\infty}\esssup\limits_{x\in\Omega}{\tilde \Psi}(Z_{N}(x)),
\end{equation}
where \({\tilde \Psi}\) denotes the support function of the closed convex set \(C_{\Psi}\coloneqq \mleft\{ z\in{\mathbb R}^{n}\mathrel{} \middle|\mathrel{} \Psi(z)\le 1\mright\}\).
\end{lemma}

We give an elementary proof of Lemma \ref{Convergence lemma in L-infty}, which is based on a definition of \({\tilde \Psi}\).

\begin{proof}
We consider the case \(C_{\infty}\coloneqq \liminf\limits_{N\to\infty}\,\mleft\lVert {\tilde \Psi}(Z_{N})\mright\rVert_{L^{\infty}(\Omega)}<\infty\), since otherwise (\ref{weakstar lower semicontinuity lemma}) is clear.
Fix arbitrary \(\varepsilon>0\). Then we may take a subsequence \(\{ Z_{N_{j}} \}_{j=1}^{\infty}\) such that 
\begin{equation}\label{bound on Psi}
\esssup\limits_{x\in\Omega}{\tilde \Psi}(Z_{N_{j}} (x))\le C_{\infty}+\varepsilon<\infty.
\end{equation}
Take arbitrary \(0\le \phi\in L^{1}(\Omega)\) and \(w\in C_{\Psi}\). 
Then with the aid of (\ref{CW ineq}), we have
\[\langle Z_{N_{j}}(x)\mid w\rangle\le C_{\infty}+\varepsilon\]
for all \(j\in{\mathbb N}\) and for a.e. \(x\in\Omega\), which yields
\begin{equation}\label{a claim; positivity of functional}
\int_{\Omega} \mleft[C_{\infty}+\varepsilon-\langle Z_{N_{j}}(x) \mid w\rangle \mright]\phi(x)\,dx\ge 0
\end{equation}
for all $j\in{\mathbb N}$. Letting \(j\to\infty\), we have
\[\int_{\Omega} \mleft[C_{\infty}+\varepsilon-\langle Z(x) \mid w\rangle \mright]\phi(x)\,dx\ge 0\]
by \(Z_{N_{j}}\overset{\ast}{\rightharpoonup} Z\) in \(L^{\infty}(\Omega,\,{\mathbb R}^{n})\).
Since \(0\le \phi\in L^{1}(\Omega)\) is arbitrary, for each \(w\in C_{\Psi}\), there exists an ${\mathcal L}^{n}$-measurable set \(U_{w}\subset \Omega\), such that \({\mathcal L}^{n}(U_{w})=0\) and
\[\langle Z(x)\mid w\rangle \le C_{\infty}+\varepsilon \quad\textrm{for all }x\in \Omega\setminus U_{w}.\]
Here we denote \({\mathcal L}^{n}\) by the \(n\)-dimensional Lebesgue measure.
Since \(C_{\Psi}\subset{\mathbb R}^{n}\) is separable, we may take a countable and dense set \(D\subset C_{\psi}\). We set an ${\mathcal L}^{n}$-measurable set
\[U\coloneqq \bigcup_{w\in D} U_{w}\subset \Omega,\]
which clearly satisfies ${\mathcal L}^{n}(U)=0$. Then we conclude that 
\[\langle Z(x)\mid w\rangle\le C_{\infty}+\varepsilon\quad \textrm{for all }x\in \Omega\setminus U,\,w\in C_{\Psi}\]
from density of \(D\subset C_{\Psi}\).
Hence by definition of \({\tilde \Psi}\), it is clear that
\[{\tilde \Psi}(Z(x))\le C_{\infty}+\varepsilon \quad\textrm{for a.e. }x\in\Omega.\]
Since \(\varepsilon>0\) is arbitrary, this completes the proof of (\ref{weakstar lower semicontinuity lemma}).
\end{proof}

\subsection{Sketches of the proofs}\label{Subsect Sketch}
We first give definitions of weak solutions to (\ref{generalized eq}). We also define weak subsolutions, and supersolutions to an equation $Lu=0$ in a bounded domain.
\begin{definition}\upshape\label{Def generalized weak sols}
Let \(\Omega\subset{\mathbb R}^{n}\) be a domain.
\begin{enumerate}
\item Let \(f\in L_{\mathrm{loc}}^{q}(\Omega)\,(n<q\le\infty)\). We say that a function \(u\in W_{\mathrm{loc}}^{1,\,\infty}(\Omega)\) is a \textrm{weak} solution to (\ref{generalized eq}), when for any bounded Lipschitz domain \(\omega\Subset\Omega\), there exists a vector field \(Z\in L^{\infty}(\omega,\,{\mathbb R}^{n})\) such that the pair \((u,\,Z)\in W^{1,\,\infty}(\omega)\times L^{\infty}(\omega,\,{\mathbb R}^{n})\) satisfies
\begin{equation}\label{general weak formulation}
\int_{\omega} \langle Z\mid \nabla\phi\rangle\,dx+\int_{\omega}\langle A(\nabla u)\mid\nabla \phi\rangle\,dx=\int_{\omega}f\phi\,dx
\end{equation}
for all \(\phi\in W_{0}^{1,\,1}(\omega)\), and
\begin{equation}\label{general subgradient}
Z(x)\in \partial \Psi(\nabla u(x))
\end{equation}
for a.e. \(x\in\omega\). Here \(A\) denotes the continuous mapping \(A\colon {\mathbb R}^{n}\ni x\mapsto \nabla_{z}W(x)\in{\mathbb R}^{n}\). For such pair \((u,\,Z)\), we say that $(u,\,Z)$ satisfies $Lu=f$ in $W^{-1,\,\infty}(\omega)$ or simply say that $u$ satisfies $Lu=f$ in $W^{-1,\,\infty}(\omega)$.
\item Assume that \(\Omega\) is bounded. A pair \((u,\,Z)\in W^{1,\,\infty}(\Omega)\times L^{\infty}(\Omega,\,{\mathbb R}^{n})\) is called a \textit{weak} subsolution to \(Lu=0\) in \(\Omega\), if it satisfies
\begin{equation}\label{weak subsolution generalized}
\int_{\Omega}\langle Z \mid \nabla\phi\rangle\,dx+\int_{\Omega}\mleft\langle A(\nabla u)\mathrel{}\middle|\mathrel{}\nabla\phi\mright\rangle\,dx\le 0
\end{equation}
for all $0\le\phi\in C_{c}^{\infty}(\Omega)$, and
\begin{equation}\label{gradient vector field generalized}
Z(x)\in\partial\Psi\mleft(\nabla u(x)\mright)\quad\textrm{for a.e. }x\in\Omega.
\end{equation}
Similarly we call a pair \((u,\,Z)\in W^{1,\,p}(\Omega)\times L^{\infty}(\Omega,\,{\mathbb R}^{n})\) a \textit{weak} supersolution \(L_{b,\,p}u=0\) in \(\Omega\), if it satisfies (\ref{gradient vector field generalized}) and 
\[\int_{\Omega}\langle Z \mid \nabla\phi\rangle\,dx+\int_{\Omega}\mleft\langle A(\nabla u)\mathrel{}\middle|\mathrel{}\nabla\phi\mright\rangle\,dx\ge 0\]
for all $0\le\phi\in C_{c}^{\infty}(\Omega)$.
For \(u\in W^{1,\,p}(\Omega)\), we simply say that \(u\) is respectively a subsolution and a supersolution to \(Lu=0\) in the \textit{weak} sense if there is \(Z\in L^{\infty}(\Omega,\,{\mathbb R}^{n})\) such that the pair \((u,\,Z)\) is a weak subsolution and a weak supersolution to \(Lu=0\) in \(\Omega\).
\end{enumerate}
\end{definition}
\begin{remark}\upshape\label{well-definedness on weak formulations}
We describe some remarks on Definition \ref{Def generalized weak sols}.
\begin{enumerate}
\item In this paper we treat a convex solution, which clearly satisfies local Lipschitz regularity. Hence it is not restrictive to assume local or global \(W^{1,\,\infty}\)-regularity for solutions in Definition \ref{Def generalized weak sols}. Also it should be noted that if a vector field \(Z\) satisfies (\ref{general subgradient}), then \(Z\) is in \(L^{\infty}\) by Remark \ref{boundedness of subdifferential}. Hence our regularity assumptions of the pair $(u,\,Z)$ involve no loss of generality.
\item Integrals in (\ref{general weak formulation}) make sense by \(Z,\,\nabla u\in L^{\infty}(\omega,\,{\mathbb R}^{n})\), \(A\in C({\mathbb R}^{n},\,{\mathbb R}^{n})\), and the continuous embedding \(W_{0}^{1,\,1}(\omega)\hookrightarrow L^{q^{\prime}}(\omega)\).
\item For a bounded domain \(\Omega\subset{\mathbb R}^{n}\), let \(u\in C^{2}(\overline{\Omega})\) satisfy
\[\nabla u(x) \not= 0\quad \textrm{for all }x\in\Omega,\textrm{ and }\]
\[Lu(x)\le 0 \quad \textrm{for all }x\in\Omega.\]
Then the pair \((u,\,\nabla_{z}\Psi(\nabla u))\in W^{1,\,p}(\Omega)\times L^{\infty}(\Omega,\,{\mathbb R}^{n})\) satisfies (\ref{weak subsolution generalized})--(\ref{gradient vector field generalized}).
For such \(u\), we simply say that \(u\) satisfies \(Lu\le 0\) in \(\Omega\) in the \textit{classical} sense.
\end{enumerate}
\end{remark}

To prove Theorem \ref{C1 theorem generalized}, we may assume that $\Omega$ is a bounded convex domain, since our argument is local. As described in Section \ref{Subsect Main theorems}, we would like to prove that a convex solution $u$ to (\ref{generalized eq}) satisfies (\ref{reduced question}) for all \(x\in \Omega\).

For the case \(x\in D\), we can show (\ref{reduced question}) by De Giorgi--Nash--Moser theory. This is basically due to the fact that the functional
\[E(z)\coloneqq \Psi(z)+W(z)\quad\textrm{for }z\in{\mathbb R}^{n}\]
satisfy the following property. For each fixed constants \(0<\mu\le M<\infty\), there exists constants \(0<\lambda\le\Lambda<\infty\) such that the estimates (\ref{uniformly elliptic on a ball in D})--(\ref{uniformly bound on a ball in D}) hold for all $z_{0},\,\zeta,\,\omega\in{\mathbb R}^{n}$ with $\mu\le \lvert z_{0}\rvert\le M$. In other words, the operator $L$ is \textit{locally uniformly elliptic outside a facet}, in the sense that for a function \(v\) the operator \(Lv\) becomes uniformly elliptic in a place where \(0<\mu\le \lvert \nabla v\rvert\le M<\infty\) holds. This ellipticity is an easy consequence of (\ref{ellipticity on W})--(\ref{bound on W}) and (\ref{degenerate ellipticity for 1 degree})--(\ref{bounds for 1 degree}). Appealing to local uniform ellipticity of the operator $L$ outside the facet and De Giorgi--Nash--Moser theory, we are able to show that a convex solution to $Lu=f$ is $C^{1,\,\alpha}$ near a neighborhood of each fixed point $x\in D$, similarly to the proof of Lemma \ref{regularity at regular point}.

For the case $x\in F$, we first make a blow-argument to construct a convex function \(u_{0}\colon{\mathbb R}^{n}\rightarrow {\mathbb R}\) satisfying \(\partial u(x)\subset \partial u_{0}(x)\), and \(Lu_{0}=0\) in \({\mathbb R}^{n}\) in the sense of Definition \ref{Def generalized weak sols}. Next we justify a maximum principle, which is described as in (\ref{strong comparison principle}), holds on each connected component of \(D\). This result enables us to apply Lemma \ref{determination of the shape of the facet}, and thus similarly in Section \ref{Subsect Liouville}, we are able to prove a Liouville-type theorem. Hence it follows that a convex solution \(u_{0}\), which is constructed by the previous blow-argument, should be constant. Finally the inclusions \(\{0\}\subset \partial u(x)\subset\partial u_{0}(x)\subset\{ 0\}\) hold, and this completes the proof of (\ref{reduced question}), i.e., \(\partial u(x)=\{ 0\}\).

For maximum principles on the equation \(Lu=0\), the proofs are almost similar to those in Section \ref{Sect Maximum Principles}. Indeed, we first recall that the operator \(A\colon {\mathbb R}^{n}\ni z_{0}\mapsto \nabla_{z} W(z_{0})\in{\mathbb R}^{n}\) satisfies strict monotonicity (\ref{strict monotonicity}). Combining with monotonicity of the subdifferential operator \(\partial\Psi\), we can easily prove a comparison principle as in Proposition \ref{comparison principle}. Also, similarly to Lemma \ref{Hopf Lemma}, we can construct classical barrier subsolutions to $Lu=0$ in an open annulus, since the operator $L$ is locally uniformly elliptic outside a facet. These results enable us to prove a maximum principle outside a facet.

We are left to justify the remaining two problems, a blow-argument and the Liouville-type theorem. To show them, we have to make use of some basic facts on a convex functional which is homogeneous of degree $1$. These fundamental results are contained in Section \ref{Subsect convex deg 1}. 

For a blow-up argument as in Section \ref{Sect Blow-argument}, we similarly define rescaled solutions. Existence of a limit of these rescaled functions are guaranteed by the Arzel\`{a}--Ascoli theorem and a diagonal argument.
By proving Lemma \ref{convergence in the distributional sense} below, we are able to demonstrate that \(u_{0}\), a limit of rescaled solutions, is a weak solution to $Lu=0$ in ${\mathbb R}^{n}$, and this finishes our blow-up argument.
\begin{lemma}\label{convergence in the distributional sense}
Let \(U\subset{\mathbb R}^{n}\) be a bounded domain. Assume that sequences of functions \(\{u_{N}\}_{N=1}^{\infty}\subset W^{1,\,\infty}(U)\) and \(\{ f_{N}\}_{N=1}^{\infty}\subset L^{q}(U)\,(n<q\le\infty)\) satisfy all of the following.
\begin{enumerate}
\item For each \(N\in{\mathbb N}\), \(u_{N}\) satisfies \(Lu_{N}=f_{N}\) in \(W^{-1,\,\infty}(U)\).
\item There exists a constant \(M>0\), independent of \(N\in{\mathbb N}\), such that
\begin{equation}\label{uniform bound generalized}
\lvert \nabla u_{N}(x)\rvert\le M\quad\textrm{for a.e. }x\in U.
\end{equation}
\item There exists a function \(u\in W^{1,\,\infty}(U)\) such that
\begin{equation}\label{a.e. conv generalized}
\nabla u_{N}(x)\rightarrow\nabla u(x)\quad \textrm{for a.e. }x\in U.
\end{equation}
\item \(f_{N}\) strongly converges to \(0\) in \(L^{q}(U)\).
\end{enumerate}
Then \(u\) satisfies \(Lu=0\) in \(W^{-1,\,\infty}(U)\).
\end{lemma}
\begin{proof}
For each \(N\in{\mathbb N}\), there exists a vector field \(Z_{N}\in L^{\infty}(U,\,{\mathbb R}^{n})\) such that
\begin{equation}
Z_{N}(x)\in\partial \Psi(\nabla u(x))\quad\textrm{for a.e. }x\in U,
\end{equation}
\begin{equation}\label{weak form generalized local}
\int_{U}\langle Z_{N}\mid \nabla\phi\rangle\,dx+\int_{U}\langle A(\nabla u_{N})\mid\nabla\phi\rangle\,dx=\int_{U}f_{N}\phi\,dx\quad\textrm{for all }\phi\in W_{0}^{1,\,1}(U).
\end{equation}
Combining the assumption \(f_{N}\to f\) in \(L^{q}(U)\) with the continuous embedding \(L^{q}(U)\hookrightarrow W^{-1,\,\infty}(U)\), we get 
\begin{equation}\label{conv of external force term}
f_{N}\rightarrow 0\quad\textrm{in }W^{-1,\,\infty}(U).
\end{equation}
By \(A\in C({\mathbb R}^{n},\,{\mathbb R}^{n})\) and (\ref{uniform bound generalized}), the vector fields \(\{A(\nabla u_{N})\}_{N=1}^{\infty}\) satisfy
\[A(\nabla u_{N}(x))\rightarrow A(\nabla u(x))\quad\textrm{for a.e. }x\in U,\]
\[\lvert A_{N}(\nabla u_{N}(x))-A(\nabla u(x))\rvert\le C\quad\textrm{for a.e. }x\in U,\]
where $C$ is independent of \(N\in{\mathbb N}\). From these and Lebesgue's dominated convergence theorem, it follows that
\begin{equation}\label{weak convergence generalized}
A(\nabla u_{N})\overset{\ast}{\rightharpoonup} A(\nabla u)\quad\textrm{in }L^{\infty}(U,\,{\mathbb R}^{n}).
\end{equation}
As mentioned in Remark \ref{boundedness of subdifferential}--\ref{well-definedness on weak formulations}, the \(\{Z_{N}\}_{N=1}^{\infty}\subset L^{\infty}(U,\,{\mathbb R}^{n})\) is bounded. Hence by \cite[Corollary 3.30]{MR2759829}, we may take a subsequence \(\{Z_{N_{j}}\}_{j=1}^{\infty}\) so that
\begin{equation}\label{weak star limit generalized}
Z_{N_{j}} \overset{\ast}{\rightharpoonup} Z\quad\textrm{in }L^{\infty}(U,\,{\mathbb R}^{n})
\end{equation}
for some \(Z\in L^{\infty}(U,\,{\mathbb R}^{n})\). By (\ref{weak form generalized local})--(\ref{weak star limit generalized}) we obtain
\[\int_{U}\langle Z\mid \phi\rangle\,dx+\int_{U}\langle A(\nabla u)\mid\nabla\phi\rangle\,dx=0\quad\textrm{for all }\phi\in W_{0}^{1,\,1}(U).\]
Now we are left to prove that
\[Z(x)\in\partial\Psi(\nabla u(x))\quad\textrm{for a.e. }x\in U.\]
By (\ref{subdifferential expression}), it suffices to show that \(Z\) satisfies
\begin{equation}\label{first problem}
{\tilde \Psi}(Z(x))\le 1,
\end{equation}
\begin{equation}\label{second problem}
\Psi(\nabla u(x))=\langle Z\mid \nabla u(x)\rangle
\end{equation}
for a.e. \(x\in U\).
Similarly, it follows that for each \(N\in{\mathbb N}\), the vector field \(Z_{N}\) satisfies
\[\mleft\{\begin{array}{rcl}
{\tilde\Psi}(Z_{N}(x))&\le &1,\\
\Psi(\nabla u_{N}(x))&=&\langle Z_{N}\mid \nabla u(x)\rangle,
\end{array}\mright.\quad\textrm{for a.e. }x\in U.\]
Hence (\ref{first problem}) is an easy consequence of Lemma \ref{Convergence lemma in L-infty}. We recall (\ref{regularity assumptions}), and thus \(\partial\Psi(z_{0})=\{\nabla_{z}\Psi(z_{0})\}\) holds for all \(z_{0}\in{\mathbb R}^{n}\setminus\{ 0\}\). Combining (\ref{a.e. conv generalized}), we can check that \(Z_{N}(x)\rightarrow Z(x)\) for a.e. \(x\in D\coloneqq \{x\in U\mid \nabla u(x)\not= 0\}\). Hence (\ref{second problem}) holds for a.e. \(x\in D\). Note that (\ref{second problem}) is clear for \(x\in U\setminus D\), and this completes the proof.
\end{proof}

We prove a Liouville-type theorem as in Theorem \ref{Liouville theorem}. In other words, for a convex solution to \(Lu=0\) in \({\mathbb R}^{n}\), we show that $F$, the facet of \(u\), would satisfy either $F=\emptyset$ or $F={\mathbb R}^{n}$. Assume by contradiction that \(F\) satisfies \(\emptyset\subsetneq F\subsetneq {\mathbb R}^{n}\). Then by Lemma \ref{determination of the shape of the facet}, we may write a convex solution \(u\) by either of (\ref{Possible Convex Solution Type 1})--(\ref{Possible Convex Solution Type 3}). However, Lemma \ref{Lemma of no longer solution} below states that $u$ is no longer a weak solution, and this completes our proof.
\begin{lemma}\label{Lemma of no longer solution}
Let $u$ be a piecewise-linear function defined as in either of (\ref{Possible Convex Solution Type 1})--(\ref{Possible Convex Solution Type 3}). Then $u$ is not a weak solution to $Lu=0$ in ${\mathbb R}^{n}$.
\end{lemma}
\begin{proof}
As in the proof of Theorem \ref{Liouville theorem}, we introduce a constant \(d>0\), and set open cubes $Q^{\prime}\subset{\mathbb R}^{n-1}$ and $Q,\,Q_{l},\,Q_{r}\subset {\mathbb R}^{n}$. By choosing sufficiently small $d>0$, we show that $u$ does not satisfy $Lu=0$ in $W^{-1,\,\infty}(Q)$. Assume by contradiction that there exists a vector field \(Z\in L^{\infty}(Q,\,{\mathbb R}^{n})\) such that the pair \((u,\,Z)\) satisfies $Lu=0$ in $W^{-1,\,\infty}(Q)$.

We first show that a function $u$ defined as in (\ref{Possible Convex Solution Type 1}) is not a weak solution. For this case, (\ref{gradient and Hessian for homogeneous}) implies that \(Z\) satisfies \(Z(x)=\nabla_{z}\Psi(e_{1})\) for a.e. \(x\in Q_{r}\). We take and fix non-negative functions \(\phi_{1}\in C_{c}^{1}((-d,\,d)),\,\phi_{2}\in C_{c}^{1}\mleft(Q^{\prime}\mright)\) such that (\ref{test function in 2 variables}) holds, and define \(\phi\in C_{c}^{1}(Q)\) by \(\phi(x_{1},\,x^{\prime})\coloneqq \phi_{1}(x_{1})\phi_{2}(x^{\prime})\) for \((x_{1},\,x^{\prime})\in (-d,\,d) \times Q^{\prime}=Q\). Testing \(\phi\) into \(Lu=0\) in \(W^{-1,\,\infty}(Q)\), we have
\begin{align*}
0&=\int_{Q_{l}}\langle Z+A(0)\mid \nabla(\phi_{1}\phi_{2})\rangle\,dx+\int_{Q_{r}}\mleft\langle \nabla_{z}\Psi(e_{1})+A(t_{1}e_{1})\mathrel{}\middle| \mathrel{} \nabla(\phi_{1}\phi_{2})\mright\rangle\,dx\\&\le \int_{Q_{l}}\Psi(\nabla(\phi_{1}\phi_{2})){\tilde\Psi}(Z(x)) \,dx\\&\quad +\phi_{1}(0)\int_{Q^{\prime}}\phi_{2}(x^{\prime})\langle \nabla_{z}\Psi(e_{1})\mid -e_{1}\rangle\,dx^{\prime}+\phi_{1}(0)\int_{Q^{\prime}}\phi_{2}(x^{\prime})\langle A(t_{1}e_{1})\mid-e_{1}\rangle\,dx^{\prime}\\&\eqqcolon I_{1}+I_{2}+I_{3}.
\end{align*}
Here we have used the Cauchy--Schwarz-type inequality (\ref{gradW at 0}) for the integral over \(Q_{l}\), and applied the Gauss--Green theorem to the integration over $Q_{r}$. For \(I_{1}\), we make use of (\ref{CW ineq})--(\ref{triangle ineq}), Fubini's theorem and (\ref{test function in 2 variables}). Then we have
\begin{align*}
I_{1}&\le \int_{Q_{l}}\phi_{1}(x_{1})\Psi(0,\,\nabla_{x^{\prime}}\phi_{2}(x^{\prime}))\,dx+\int_{Q_{l}}\phi_{1}^{\prime}(x_{1})\phi_{2}(x^{\prime})\Psi(e_{1})\,dx\\&\le \phi_{1}(0)\mleft(d\cdot \lVert \Psi(0,\,\nabla_{x^{\prime}}\phi_{2})\rVert_{L^{1}(Q^{\prime})}+\Psi(e_{1})\lVert \phi_{2}\rVert_{L^{1}(Q^{\prime})}\mright),
\end{align*}
where \(\nabla_{x^{\prime}}\phi_{2}\coloneqq (\partial_{x_{2}}\phi_{2},\,\dots,\,\partial_{x_{n}}\phi_{2})\).
For \(I_{2}\), recalling Euler's identity (\ref{Euler's id}), we get \(I_{2}=-\phi_{1}(0)\Psi(e_{1})\lVert \phi_{2}\rVert_{L^{1}(Q^{\prime})}\). We set a constant \(\mu\coloneqq \langle A(t_{1}e_{1})\mid e_{1}\rangle\), which is positive by (\ref{claim for gradW}). Then we obtain
\[I_{1}+I_{2}+I_{3}\le \phi_{1}(0)\mleft(d\cdot \lVert \Psi(0,\,\nabla_{x^{\prime}}\phi_{2})\rVert_{L^{1}(Q^{\prime})}-\mu\lVert \phi_{2}\rVert_{L^{1}(Q^{\prime})}\mright).\]
Choosing \(d=d(\mu,\,\Psi,\,\phi_{2})>0\) sufficiently small, we have \(0\le I_{1}+I_{2}+I_{3}<0\), which is a contradiction. Similarly we can deduce that \(u\) defined as in (\ref{Possible Convex Solution Type 3}) does not satisfy \(Lu=0\) in \(W^{-1,\,\infty}(Q)\), since it suffices to restrict \(d<l_{0}\). For the remaining case (\ref{Possible Convex Solution Type 2}), we have already known that
\[Z(x)=\mleft\{\begin{array}{cc}
\nabla_{z}\Psi(e_{1}) & \textrm{for a.e. }x\in Q_{r},\\ \nabla_{z}\Psi(-e_{1}) & \textrm{for a.e. }x\in Q_{l}\end{array} \mright.\]
by definition of \(Z\) and (\ref{gradient and Hessian for homogeneous}).
We set two constants \(\mu_{1}\coloneqq \langle A(t_{1}e_{1})\mid e_{1}\rangle,\,\mu_{2}\coloneqq \langle A(-t_{2}e_{1})\mid -e_{1}\rangle\), both of which are positive by (\ref{claim for gradW}).
Testing the same function \(\phi\in C_{c}^{1}(Q)\) into \(Lu=0\) in \(W^{-1,\,\infty}(Q)\), we obtain
\begin{align*}
0&=\int_{Q_{l}}\mleft\langle \nabla_{z}\Psi(-e_{1})+A(-t_{2}e_{1})\mathrel{}\middle| \mathrel{} \nabla(\phi_{1}\phi_{2})\mright\rangle\,dx+\int_{Q_{r}}\mleft\langle \nabla_{z}\Psi(e_{1})+A(t_{1}e_{1})\mathrel{}\middle| \mathrel{} \nabla(\phi_{1}\phi_{2})\mright\rangle\,dx\\&=\int_{Q^{\prime}}\phi_{1}(0)\phi_{2}(x^{\prime})\mleft\langle \nabla_{z}\Psi(-e_{1})+A(-t_{2}e_{1}) \mathrel{}\middle| \mathrel{} e_{1}\mright\rangle\,dx^{\prime}+\int_{Q^{\prime}}\phi_{1}(0)\phi_{2}(x^{\prime})\mleft\langle \nabla_{z}\Psi(e_{1})+A(t_{1}e_{1}) \mathrel{}\middle| \mathrel{} -e_{1}\mright\rangle\,dx^{\prime}\\&=-\phi_{1}(0)\mleft(\Psi(e_{1})+\Psi(-e_{1})+\mu_{1}+\mu_{2}\mright)\int_{Q^{\prime}}\phi_{2}(x^{\prime})\,dx^{\prime}<0,
\end{align*}
which is a contradiction. Here we have used the Gauss--Green theorem and Euler's identity (\ref{Euler's id}). This completes the proof.
\end{proof}\section*{Acknowledgement}
The first author is partly supported by the Japan Society for the Promotion of Science through grants Kiban A (No. 19H00639). Challenging Pioneering Research (Kaitaku) (No. 18H05323), Kiban A (No. 17H01091). 

\appendix
\section{Proofs for a few basic facts}\label{Appendix A}
In this section, we give proofs for a few basic facts used in this paper for completeness.
\subsection{A Poincar\'{e}-type inequality}
We give a precise proof of Lemma \ref{Poincare-type lemma for differential quotient}, a Poincar\'{e}-type inequality for difference quotients of functions in \(W_{0}^{1,\,p}\,(1\le p<\infty)\). This result is used in the proof of Lemma \ref{regularity at regular point}. The proof of Lemma \ref{Poincare-type lemma for differential quotient} is essentially a modification of that of the Poincar\'{e} inequality for the Sobolev space \(W_{0}^{1,\,p}\) \cite[Proposition 3.10]{MR3099262}.
\begin{lemma}\label{Poincare-type lemma for differential quotient}
Let \(\Omega\subset {\mathbb R}^{n}\) be a bounded open set and \(1\le p<\infty\). For all \(u\in W_{0}^{1,\,p}(\Omega),j\in\{\,1,\,\dots,\,n\,\},\,\,h\in{\mathbb R}\setminus\{0\}\), we have
\begin{equation}\label{Poincare-type estimate for differential quotient}
\lVert \Delta_{j,\,h}u\rVert_{L^{p}(\Omega)}\le \lVert \nabla u\rVert_{L^{p}(\Omega)}.
\end{equation}
Here \(\Delta_{j,\,h}u\) is defined by
\[\Delta_{j,\,h}u(x)\coloneqq \frac{\overline{u}(x+he_{j})-u(x)}{h}\quad\textrm{for }x\in\Omega.\]
\end{lemma}
Before the proof of Lemma \ref{Poincare-type lemma for differential quotient}, we note that \(\Delta_{j,\,h}u(x)\) makes sense for a.e. \(x\in\Omega\) by the zero extension of \(u\in W_{0}^{1,\,p}(U)\).
That is, for a given $u\in W_{0}^{1,\,p}(U)$, we set $\overline{u}\in W^{1,\,p}({\mathbb R}^{n})$ by
\begin{equation}\label{explicit extension by zero}
\overline{u}(x)\coloneqq\left\{\begin{array}{cc}u(x)& x\in U, \\ 0 & x\in{\mathbb R}^{n}\setminus U.\end{array}\right.
\end{equation}
\begin{proof}
We fix \(j\in\{\,1,\,\dots,\,n\,\},\,\,h\in{\mathbb R}\setminus\{0\}\).
We first note that the operator \(\Delta_{j,\,h}\colon W_{0}^{1,\,p}(U)\rightarrow L^{p}(U)\) is bounded, since for all \(u\in W_{0}^{1,\,p}(U)\) we have
\begin{align*}
\lVert \Delta_{j,\,h}u\rVert_{L^{p}(U)}&\le \frac{1}{\lvert h\rvert}\mleft[\mleft(\int_{U}\lvert \overline{u}(x+h)\rvert^{p}\,dx\mright)^{1/p}+\mleft(\int_{U}\lvert u(x)\rvert^{p}\,dx\mright)^{1/p}\mright]\\&\le \frac{2}{\lvert h\rvert}\lVert u\rVert_{L^{p}(U)}\le \frac{C(p,\,U)}{\lvert h\rvert}\lVert \nabla u\rVert_{L^{p}(U)}
\end{align*}
by the Minkowski inequality and the Poincar\'{e} inequality. Here \(\overline{u}\in W^{1,\,p}({\mathbb R}^{n})\) is defined as in (\ref{explicit extension by zero}).
Hence by a density argument, it suffices to check that (\ref{Poincare-type estimate for differential quotient}) holds true for all \(u\in C_{c}^{\infty}(U)\).
Let \(u\in C_{c}^{\infty}(U)\).
Then for all \(x\in U\), we have 
\begin{align*}
\lvert \overline{u}(x+he_{j})-u(x)\rvert&=\mleft\lvert\int_{0}^{1}\mleft\langle\nabla \overline{u}(x+the_{j})\mathrel{}\middle|\mathrel{}he_{j}\mright\rangle\,dt\mright\rvert\\&\le\lvert h\rvert \int_{0}^{1}\lvert\nabla \overline{u}(x+the_{j})\rvert \,dt\le \lvert h\rvert \mleft(\int_{0}^{1}\lvert\nabla \overline{u}(x+the_{j})\rvert^{p}\,dt\mright)^{1/p}
\end{align*}
by the Cauchy-Schwarz inequality and H\"{o}lder's inequality. 
From this estimate we get
\begin{align*}
\lVert \Delta_{j,\,h}u\rVert_{L^{p}(U)}^{p}&\le \int_{\Omega}\int_{0}^{1}\lvert \nabla \overline{u}(x+the_{j})\rvert^{p}\,dt\,dx\\&=\int_{0}^{1}\underbrace{\int_{U}\lvert \nabla \overline{u}(x+the_{j})\rvert^{p}\,dx}_{\le \lVert \nabla u\rVert_{L^{p}(U)}^{p}}\,dt\quad (\textrm{by Fubini's theorem})
\\&\le \lVert \nabla u\rVert_{L^{p}(U)}^{p}.
\end{align*}
Hence we obtain (\ref{Poincare-type estimate for differential quotient}) for all \(u\in C_{c}^{\infty}(U)\), and this completes the proof.
\end{proof}
\subsection{Convex analysis}
Lemma \ref{coercivity lemma for convex function} is used in the proof of Lemma \ref{regularity at regular point} for a justification of local \(W^{2,\,2}\)-regularity of a convex weak solution outside of the facet.
\begin{lemma}\label{coercivity lemma for convex function}
Let $u$ be a real-valued convex function in a convex domain $\Omega\subset{\mathbb R}^{n}$.
Assume that \(x_{1},\,x_{2}\in\Omega\) satisfy \(x_{1}\neq x_{2}\), and set
\(d\coloneqq \lvert x_{2}-x_{1}\rvert>0,\,\nu\coloneqq d^{-1}(x_{2}-x_{1})\). Then for all \(z_{2}\in\partial u(x_{2})\), we have
\begin{equation}\label{An estimate by convexity}
\langle z_{2} \mid \nu \rangle\ge \frac{u(x_{2})-u(x_{1})}{d}.
\end{equation}
\end{lemma}
\begin{proof}
By \(z_{2}\in\partial u(x_{2})\), we have a subgradient inequality
\[u(x)\ge u(x_{2})+\langle z_{2}\mid x-x_{2}\rangle\]
for all \(x\in\Omega\). Substituting \(x\coloneqq x_{1}=x_{2}-d\nu\in\Omega\), we obtain
\[u(x_{1})\ge u(x_{2})-d\langle z_{2}\mid \nu\rangle,\]
which yields (\ref{An estimate by convexity}).
\end{proof}
\begin{remark}\upshape
Instead of subgradient inequalities, we are able to show (\ref{An estimate by convexity}) by monotonicity of \(\partial u\). For each fixed \(x_{1},\,x_{2}\in\Omega\) with \(x_{1}\neq x_{2}\), we may take and fix \(x_{3}\coloneqq x_{1}+t(x_{2}-x_{1})\) for some \(0<t<1\) and \(z_{3}\in\partial u(x_{3})\) such that
\begin{equation}\label{mean value theorem for non-smooth convex function}
u(x_{2})-u(x_{1})=\langle z_{3}\mid x_{2}-x_{1}\rangle,
\end{equation}
with the aid of the mean value theorem for non-smooth convex functions \cite[Theorem D.6]{MR3887613}. \(x_{2}-x_{1}=d\nu\) is clear by definitions of \(d,\,\nu\). Noting \(x_{2}-x_{3}=(1-t)d\nu\), we can check that
\[\langle z_{2}-z_{3}\mid \nu \rangle=\frac{1}{(1-t)d}\langle z_{2}-z_{3} \mid x_{2}-x_{3}\rangle\ge 0\]
by monotonicity of \(\partial u\). Combining these results with (\ref{mean value theorem for non-smooth convex function}), we obtain
\[u(x_{2})-u(x_{1})=d\langle z_{3}\mid \nu\rangle\le d\langle z_{2}\mid \nu\rangle,\]
which yields (\ref{An estimate by convexity}).
\end{remark}
The following lemma is used in the proof of Proposition \ref{limit of rescaled solutions}.
\begin{lemma}\label{a.e. pointwise convergence for convex functions}
Let \(U\subset {\mathbb R}^{n}\) be a convex open set, and let \(\{ u_{N}\}_{N=1}^{\infty}\) be a sequence of real-valued convex functions in \(U\). Assume that this sequence is uniformly Lipschitz. In other words, there is a constant \(L>0\) independent of \(N\in{\mathbb N}\) such that
\begin{equation}\label{uniform bound}
\lvert u_{N}(x)-u_{N}(y)\rvert\le L\lvert x-y\rvert\quad\textrm{for all }x,\,y\in U.
\end{equation}
If there exists a function \(u_{\infty}\colon U\rightarrow{\mathbb R}\) such that
\begin{equation}\label{pointwise conv}
u_{N}(x)\rightarrow u_{\infty}(x)\quad\textrm{for all }x\in U,
\end{equation}
then we have \(\nabla u_{N}(x)\rightarrow \nabla u_{\infty}(x)\) for a.e. \(x\in U\).
\end{lemma}
\begin{remark}\upshape\upshape
From (\ref{uniform bound})--(\ref{pointwise conv}), it is easy to show that \(u_{\infty}\) is also convex, \(u_{N}\to u_{\infty}\) uniformly in \(U\), and
\[\lvert u_{\infty}(x)-u_{\infty}(y)\rvert\le L\lvert x-y\rvert\quad\textrm{for all }x,\,y\in U.\]
\end{remark}
Our proof of Lemma \ref{a.e. pointwise convergence for convex functions} is inspired by \cite[Lemma A.3]{MR1671993}.
\begin{proof}
We define \({\mathcal L}^{n}\)-measurable sets
\[P_{N}\coloneqq \{x\in U\mid u_{N} \textrm{ is not differentiable at }x\}\textrm{ for }N\in{\mathbb N}\cup\{\infty\}.\]
Clearly \(P_{N}\,(N\in{\mathbb N}\cup\{\infty\})\) satisfies \({\mathcal L}^{n}(P_{N})=0\) by Lipschitz continuity of $u_{N}$, and therefore the \({\mathcal L}^{n}\)-measurable set \[P\coloneqq \bigcup\limits_{N\in {\mathbb N}\cup\{\infty\}}P_{N}\subset U\] also satisfies \({\mathcal L}^{n}(P)=0\). We claim that 
\begin{equation}\label{pointwise convergence without a null-set}
\nabla u_{N}(x_{0})\rightarrow \nabla u_{\infty}(x_{0})\quad\textrm{for all }x_{0}\in U\setminus P.
\end{equation}
We take and fix arbitrary \(x_{0}\in U\setminus P\). We note that \(\nabla u_{N}(x_{0})\) exists for each \(N\in{\mathbb N}\) since \(x_{0}\not\in P_{N}\), and we obtain
\[\sup\limits_{N\in{\mathbb N}}\,\lvert \nabla u_{N}(x_{0})\rvert\le L\]
with the aid of (\ref{uniform bound}). Hence it suffices to check that, if a subsequence \(\{u_{N_{k}}\}_{k}\subset \{u_{N}\}_{N}\) satisfies
\begin{equation}\label{full seq conv}
\nabla u_{N_{k}}(x_{0})\rightarrow v\,(k\to\infty)\quad \textrm{for some }v\in {\mathbb R}^{N},\end{equation}
then \(v=\nabla u_{\infty}(x_{0})\).
Since \(x_{0}\not\in P_{N_{k}}\) and therefore \(\partial u_{N_{k}}(x_{0})=\{\nabla u_{N_{k}}(x_{0})\}\) for each \(k\in{\mathbb N}\), we easily get 
\[u_{N_{k}}(x)\ge u_{N_{k}}(x_{0})+\langle\nabla u_{N_{k}}(x_{0})\mid x-x_{0} \rangle\quad \textrm{for all }x\in U,\, k\in{\mathbb N}.\]
Letting \(k\to\infty\), we have
\[u_{\infty}(x)\ge u_{\infty}(x_{0})+\langle v\mid x-x_{0} \rangle\quad \textrm{for all }x\in U\]
by (\ref{pointwise conv}) and (\ref{full seq conv}).
This means that \(v\in\partial u_{\infty}(x_{0})\).
Note again that \(x_{0}\not\in P_{\infty}\) and therefore \(\partial u_{\infty}(x_{0})=\{\nabla u_{\infty}(x_{0})\}\), which yields \(v=\nabla u_{\infty}(x_{0})\). This completes the proof of (\ref{pointwise convergence without a null-set}).\end{proof}
\subsection{Convex functionals}\label{Subsect convex deg 1}
We prove some basic property of convex functionals \(\Psi\) and \(W\) in Section \ref{Sect Generalization}.
\begin{lemma}\label{lemma strict monotonicity}
Let \(W\) be a convex function which satisfies (\ref{regularity assumptions})-(\ref{ellipticity on W}) and (\ref{gradW at 0}). Then the mapping \(A\colon {\mathbb R}^{n}\ni z\mapsto \nabla W(z)\in {\mathbb R}^{n}\) satisfies strict monotonicity (\ref{strict monotonicity}).
\end{lemma}
\begin{proof}
We take arbitrary \(z_{1},\,z_{2}\in{\mathbb R}^{n}\) with \(z_{1}\not= z_{2}\) and define a line segment \(L\coloneqq \mleft\{z_{1}+t(z_{2}-z_{1})\in{\mathbb R}^{n}\mathrel{}\middle|\mathrel{}0\le t\le 1\mright\}.\)

We first consider the case \(0\not\in L\). Then there exist constants \(0<\mu\le M<\infty\) such that \(\mu\le \lvert z_{0}\rvert\le M\) holds for all \(z_{0}\in L\). Here we can take a constant \(\gamma>0\) such that (\ref{ellipticity on W}) holds for all \(z_{0}\in L\). Then by \(W\in C^{2}({\mathbb R}^{n}\setminus\{ 0\})\), we have
\[\mleft\langle A(z_{1})-A(z_{2})\mathrel{}\middle|\mathrel{}z_{2}-z_{1} \mright\rangle=\int_{0}^{1}\mleft\langle\nabla_{z}^{2}W(z_{1}+t(z_{2}-z_{1}))(z_{2}-z_{1})\mathrel{}\middle|\mathrel{} z_{2}-z_{1}\mright\rangle\,dt\ge \gamma \lvert z_{2}-z_{1}\rvert^{2}>0.\]

To consider the remaining case \(0\in L\), it suffices to show (\ref{claim for gradW}).
Indeed, the assumption \(0\in L\) allows us to write \(z_{1}=-l_{1}\nu,\,z_{2}=l_{2}\nu\) for some unit vector \(\nu\) and some constants \(l_{1},\,l_{2}\ge 0\). Under this notation, we obtain
\[\langle A(z_{2})-A(z_{1})\mid z_{2}-z_{1}\rangle=\langle A(l_{2}\nu)\mid (l_{1}+l_{2})\nu\rangle+\langle A(-l_{1}\nu)\mid -(l_{1}+l_{2})\nu\rangle>0\]
by (\ref{claim for gradW}). Here we note that at least one of \(l_{1},\,l_{2}\) is positive since \(l_{1}+l_{2}=\lvert z_{2}-z_{1}\rvert>0\).

We prove (\ref{claim for gradW}) to complete the proof.
Let \(z\in{\mathbb R}^{n}\setminus\{0\}\). Then we obtain
\[d_{N}\coloneqq \mleft\langle A(z/2^{N-1})-A(z/2^{N})\mathrel{}\middle|\mathrel{} z \mright\rangle>0\]
for each \(N\in{\mathbb N}\), since we have already shown (\ref{strict monotonicity}) for the case \(0\not\in L\). By definition of \(d_{j}\,(j\in{\mathbb N})\), it is clear that
\[\mleft\langle A(z)-A(z/2^{N})\mathrel{}\middle|\mathrel{} z \mright\rangle=d_{1}+\dots+d_{N}\ge d_{1}.\]
Letting \(N\to\infty\), we obtain \(\langle A(z)\mid z\rangle\ge d_{1}>0\) by $A\in C({\mathbb R}^{n},\,{\mathbb R}^{n})$.
\end{proof}
We precisely prove (\ref{triangle ineq})--(\ref{subdifferential expression}) in Lemma \ref{basic lemma of convex deg 1}. See also \cite[Section 1.3]{MR2033382} and \cite[\S 13]{MR1451876} as related items.
\begin{lemma}\label{basic lemma of convex deg 1}
Let \(\Psi\colon {\mathbb R}^{n}\rightarrow [0,\,\infty)\) be a convex function which is positively homogeneous of degree $1$.
\begin{enumerate}
\item \(\Psi\) satisfies the triangle inequality (\ref{triangle ineq}).
\item Assume that \(\zeta\in{\mathbb R}^{n}\) satisfies \({\tilde \Psi}(\zeta)<\infty\). Then the Cauchy--Schwarz-type inequality (\ref{CW ineq}) holds.
\item The subdifferential operator \(\partial\Psi\) is given by (\ref{subdifferential expression}).
\end{enumerate}
\end{lemma}
\begin{proof}
By convexity of \(\Psi\) and (\ref{hom of deg 1}), \(\Psi\) satisfies \[\frac{\Psi(z_{1}+z_{2})}{2}=\Psi\mleft(\frac{z_{1}+z_{2}}{2}\mright)\le \frac{\Psi(z_{1})+\Psi(z_{2})}{2}\quad\textrm{for all }z_{1},\,z_{2}\in{\mathbb R}^{n},\]
which yields (\ref{triangle ineq}).

We next show the Cauchy--Schwarz inequality (\ref{CW ineq}).
Let \(z\in{\mathbb R}^{n}\). If \(\Psi(z)>0\), then we have
\[\langle z\mid \zeta \rangle=\Psi(z)\mleft\langle \frac{z}{\Psi(z)}\mathrel{}\middle|\mathrel{}\zeta \mright\rangle\le \Psi(z){\tilde \Psi}(\zeta)\]
by \(z/\Psi(z)\in C_{\Psi}\). For the case \(\Psi(z)=0\), we note that \(\lambda z\in C_{\Psi}\) for all \(\lambda>0\). Hence it follows that
\[\langle z \mid \zeta\rangle=\frac{\langle \lambda z\mid \zeta \rangle}{\lambda}\le \frac{{\tilde \Psi}(w)}{\lambda}\]
for all \(\lambda>0\). By \({\tilde \Psi}(\zeta)<\infty\), we obtain \(\langle z\mid \zeta\rangle\le 0=\Psi(z){\tilde \Psi}(\zeta)\). This completes the proof of (\ref{CW ineq}).

Finally we prove (\ref{subdifferential expression}). Let \(z_{0}\in{\mathbb R}^{n}\) be arbitrarily fixed. Assume that \(\zeta\in{\mathbb R}^{n}\) satisfies \({\tilde\Psi}(\zeta)\le 1\) and \(\Psi(z_{0})=\langle z_{0}\mid \zeta\rangle\). Then by combining these assumptions with (\ref{CW ineq}), we have 
\begin{align*}
\Psi(z)&\ge \Psi(z){\tilde\Psi}(\zeta)\\& \ge \langle z\mid\zeta\rangle=\langle z_{0}\mid\zeta\rangle+\langle z-z_{0}\mid\zeta\rangle\\&=\Psi(z_{0})+\langle \zeta\mid z-z_{0}\rangle
\end{align*}
for all \(z\in{\mathbb R}^{n}\). Hence \(\zeta\in\partial\Psi(z_{0})\). Conversely, if \(\zeta\in\partial\Psi(z_{0})\), then we have the subgradient inequality
\begin{equation}\label{subgradient ineq deg 1}
\Psi(z)\ge \Psi(z_{0})+\langle\zeta\mid z-z_{0} \rangle\quad\textrm{for all }z\in{\mathbb R}^{n}.
\end{equation}
By testing \(kz_{0}\) into (\ref{subgradient ineq deg 1}), where \(k\in[0,\,\infty)\) is arbitrary, we have
\begin{equation}
(k-1)\Psi(z_{0})=\Psi(kz_{0})-\Psi(z_{0})\ge \langle\zeta\mid (k-1)z_{0}\rangle=(k-1)\langle\zeta\mid z_{0}\rangle.
\end{equation}
If we let \(0\le k<1\) so that \(k-1<0\), then we have \(\Psi(z_{0})\le \langle\zeta\mid z_{0}\rangle\). Similarly, letting \(1<k<\infty\), we have \(\Psi(z_{0})\ge \langle\zeta\mid z_{0}\rangle\). Hence we obtain \(\Psi(z_{0})=\langle\zeta\mid z_{0}\rangle\). Combining with (\ref{subgradient ineq deg 1}), we have 
\[\langle z\mid \zeta\rangle\le \Psi(z)\quad\textrm{for all }z\in{\mathbb R}^{n},\]
which yields \({\tilde\Psi}(\zeta)\le 1\) by definition of \({\tilde\Psi}\).
This completes the proof of (\ref{subdifferential expression}).
\end{proof}
\bibliographystyle{plain}

\end{document}